\documentclass[english]{article}
\usepackage{fullpage}
\usepackage{mathtools}
\usepackage{amsmath}
\usepackage{amsthm}
\usepackage{amssymb,todonotes}
\usepackage[hidelinks, hypertexnames=false]{hyperref}
\usepackage{tikz}
\usepackage{float}
\usetikzlibrary{positioning}

\makeatletter
\numberwithin{equation}{section}
\numberwithin{figure}{section}
\theoremstyle{plain}
\newtheorem{thm}{\protect\theoremname}[section]
\theoremstyle{definition}
\newtheorem{problem}[thm]{\protect\problemname}
\theoremstyle{plain}
\newtheorem{cor}[thm]{\protect\corollaryname}
\theoremstyle{plain}
\newtheorem{conj}[thm]{\protect\conjecturename}
\theoremstyle{definition}
\newtheorem{defn}[thm]{\protect\definitionname}
\theoremstyle{plain}
\newtheorem{lem}[thm]{\protect\lemmaname}

\makeatother

\usepackage{babel}
\providecommand{\corollaryname}{Corollary}
\providecommand{\conjecturename}{Conjecture}
\providecommand{\definitionname}{Definition}
\providecommand{\lemmaname}{Lemma}
\providecommand{\problemname}{Problem}
\providecommand{\theoremname}{Theorem}

\begin{document}
\title{Ramsey numbers of sparse digraphs}
\author{Jacob Fox\thanks{Department of Mathematics, Stanford University, Stanford, CA 94305, USA. Email: \url{jacobfox@stanford.edu}. Research supported by a Packard Fellowship and by NSF award  DMS-185563.} \and Xiaoyu He\thanks{Department of Mathematics, Stanford University, Stanford, CA 94305, USA. Email: \url{alkjash@stanford.edu}. Research supported by NSF GRFP Grant DGE-1656518.} \and Yuval Wigderson\thanks{Department of Mathematics, Stanford University, Stanford, CA 94305, USA. Email: \url{yuvalwig@stanford.edu}. Research supported by NSF GRFP Grant DGE-1656518.}}
\date{}

\maketitle
\global\long\def\va{\overrightarrow{va}}%

\global\long\def\RR{\mathbb{R}}%

\global\long\def\QQ{\mathbb{Q}}%

\global\long\def\HH{\mathbb{H}}%

\global\long\def\E{\mathbb{E}}%

\global\long\def\Var{\operatorname{Var}}%

\global\long\def\CC{\mathbb{C}}%

\global\long\def\NN{\mathbb{N}}%

\global\long\def\ZZ{\mathbb{Z}}%

\global\long\def\GG{\mathbb{G}}%

\global\long\def\BB{\mathbb{B}}%

\global\long\def\DD{\mathbb{D}}%

\global\long\def\cL{\mathcal{L}}%

\global\long\def\supp{\operatorname{supp}}%

\global\long\def\one{\boldsymbol{1}}%

\global\long\def\range#1{\left[#1\right]}%

\global\long\def\d{\operatorname{d}}%

\global\long\def\falling#1#2{\left(#1\right)_{#2}}%

\global\long\def\f{\mathbf{f}}%

\global\long\def\im{\operatorname{im}}%

\global\long\def\sp{\operatorname{span}}%

\global\long\def\rank{\operatorname{rank}}%

\global\long\def\sign{\operatorname{sign}}%

\global\long\def\mod{\operatorname{mod}}%

\global\long\def\id{\operatorname{id}}%

\global\long\def\disc{\operatorname{disc}}%

\global\long\def\lindisc{\operatorname{lindisc}}%

\global\long\def\tr{\operatorname{tr}}%

\global\long\def\adj{\operatorname{adj}}%

\global\long\def\Unif{\operatorname{Unif}}%

\global\long\def\Po{\operatorname{Po}}%

\global\long\def\Bin{\operatorname{Bin}}%

\global\long\def\Ber{\operatorname{Ber}}%

\global\long\def\Exp{\operatorname{Exp}}%

\global\long\def\Geom{\operatorname{Geom}}%

\global\long\def\sat{\operatorname{sat}}%

\global\long\def\Hom{\operatorname{Hom}}%

\global\long\def\vol{\operatorname{vol}}%

\global\long\def\comp{\operatorname{comp}}%

\global\long\def\depth{\operatorname{depth}}%

\global\long\def\floor#1{\left\lfloor #1\right\rfloor }%

\global\long\def\ceil#1{\left\lceil #1\right\rceil }%

\global\long\def\cond{\,\middle|\,}%

\let\polishL\L

\global\long\def\L{\mathcal{L}}%

\DeclareRobustCommand{\L}{\ifmmode{\mathcal{L}}\else\polishL\fi}

\global\long\def\randS{\boldsymbol{S}}%

\global\long\def\S{\mathcal{S}}%

\global\long\def\Sm#1{\mathcal{S}_{#1}}%

\global\long\def\ord{\mathcal{O}}%

\global\long\def\ordm#1{\mathcal{O}_{#1}}%

\global\long\def\cordm#1#2{\mathcal{O}_{#2}^{#1}}%

\global\long\def\cSm#1#2{\mathcal{S}_{#2}^{#1}}%

\global\long\def\ext#1{\mathcal{O}^{\mathrm{ext}}\left(#1\right)}%

\global\long\def\randN{\boldsymbol{N}}%

\global\long\def\randNt{\boldsymbol{N_{2}}}%

\global\long\def\randNk{\boldsymbol{N_{k}}}%

\global\long\def\randX{\boldsymbol{X}}%

\global\long\def\coG{\overline{G}}%

\global\long\def\Sext#1{\mathcal{S}^{\mathrm{ext}}\left(#1\right)}%

\global\long\def\randl{\boldsymbol{\lambda}}%

\global\long\def\randz{\boldsymbol{z}}%

\global\long\def\randR{\boldsymbol{R}}%

\global\long\def\randL{\boldsymbol{L}}%

\global\long\def\randQ{\boldsymbol{Q}}%

\global\long\def\TT{\textnormal{TT}}%

\global\long\def\ovar#1{\overrightarrow{#1}}%
\usetikzlibrary{patterns}

\tikzstyle{vertex} = [draw, circle, fill = black!5, minimum size=11pt, inner sep=1pt]
\tikzstyle{medium vertex} = [draw, circle, fill = black!5, minimum size=15pt, inner sep=0pt]
\tikzstyle{small vertex} = [draw, circle, fill = black!5, minimum size=5pt, inner sep=0.5pt]
\tikzstyle{small yellow vertex} = [draw, circle, fill = yellow!50, minimum size=5pt, inner sep=0.5pt]
\tikzstyle{large dashed vertex} = [draw=blue!30, circle, minimum size=42.5pt, inner sep=0pt, line width=0.25mm, dashed]
\tikzstyle{enormous dashed vertex} = [draw=black!50, circle, minimum size=80pt, inner sep=0pt, line width=0.25mm, dashed]
\tikzstyle{large dashed red vertex} = [draw=red!50, fill=red!20, circle, minimum size=40pt, inner sep=0pt, line width=0.5mm, dashed]
\tikzstyle{black edge} = [draw, line width=1.2pt, ->, black]
\tikzstyle{black thin undirected edge} = [draw, line width=0.6pt, -, black]
\tikzstyle{black undirected edge} = [draw, line width=1.2pt, -, black]

\tikzstyle{red edge} = [draw, line width=1.2pt, ->, red]
\tikzstyle{red undirected edge} = [draw, line width=1.2pt, -, red]
\tikzstyle{blue undirected edge} = [draw, line width=1.2pt, -, blue]
\tikzstyle{blue thick undirected edge} = [draw, line width=3pt, -, blue]

\tikzstyle{red dashed ellipse} = [draw=red!50, circle, x radius=5mm, y radius=20mm, fill=red!20, line width=0.5mm, dashed]

\begin{abstract}
Burr and Erd\H{o}s in 1975 conjectured, and Chv\'atal, R\"odl, Szemer\'edi and Trotter later proved, that the Ramsey number of any bounded degree graph is linear in the number 
of vertices. In this paper, we disprove the natural directed analogue of the Burr--Erd\H os conjecture, answering a question of Buci\'c, Letzter, and Sudakov. If $H$ is an acyclic digraph, the \emph{oriented Ramsey number} of $H$, denoted $\ovar{r_{1}}(H)$,
is the least $N$ such that every tournament on $N$ vertices contains
a copy of $H$. We show that for any $\Delta \geq 2$ and any sufficiently large $n$, there exists
an acyclic digraph $H$ with $n$ vertices and maximum degree $\Delta$ such
that 
\[
\ovar{r_{1}}(H)\ge n^{\Omega(\Delta^{2/3}/ \log^{5/3} \Delta)}.
\]
This proves that $\ovar{r_{1}}(H)$ is not always linear in the number of vertices for bounded-degree $H$.
On the other hand, we show that $\ovar{r_{1}}(H)$ is nearly linear in the number of vertices for typical bounded-degree acyclic digraphs $H$, and obtain linear or nearly linear bounds for several natural families of bounded-degree acyclic digraphs. 

For multiple colors, we prove a quasi-polynomial upper bound $\ovar{r_{k}}(H)=2^{(\log n)^{O_{k}(1)}}$
for all bounded-degree acyclic digraphs $H$ on $n$ vertices, where $\ovar{r_k}(H)$ is the least $N$ such that every $k$-edge-colored tournament on $N$ vertices contains a monochromatic copy of $H$. For $k\ge2$ and $n\geq 4$, we exhibit an acyclic digraph $H$ with $n$ vertices and maximum degree $3$ such that $\ovar{r_{k}}(H)\ge n^{\Omega(\log n/\log\log n)}$, showing that 
these Ramsey numbers can grow faster than any polynomial in the number of vertices. 

\end{abstract}

\section{Introduction}

The $k$-color Ramsey number of a (simple undirected) graph $H$, denoted $r_k(H)$, is the minimum $N$ such that every $k$-edge-coloring of the complete graph $K_N$ contains a monochromatic copy of $H$. Broadly speaking, the main question in graph Ramsey theory is to understand how $r_k(H)$ depends on $H$ and $k$. The most well-studied case is that of two colors, $k=2$. For $H$ on $n$ vertices, it is known \cite{Conlon, Sudakov} that $r_2(H)$ grows exponentially in $n$ if and only if $H$ has $\Omega(n^2)$ edges.

However, it has long been observed that the Ramsey number of a sparse graph $H$ is much smaller than exponential in $|V(H)|$. In their foundational paper on the topic, Burr and Erd\H os \cite{BuEr} conjectured that this phenomenon is quite general and that any sparse graph has linear Ramsey number. Here, the appropriate notion of sparsity is \emph{degeneracy:} $H$ is said to be $d$-degenerate if every subgraph of $H$ has a vertex of degree at most $d$, and the degeneracy of $H$ is the minimum $d$ such that $H$ is $d$-degenerate. Burr and Erd\H os conjectured that $r_k(H) = O_{k,d}(n)$ for any $n$-vertex graph $H$ with degeneracy $d$. Here and throughout we use the standard asymptotic notation where the implicit constant is allowed to depend only on the subscripts of $O(\cdot)$. Major progress towards this conjecture was made by Chv\'atal, R\"odl, Szemer\'edi, and Trotter \cite{ChRoSzTr}, who proved the Burr--Erd\H os conjecture under the stronger assumption that $H$ has bounded degree (rather than bounded degeneracy), that is, that $r_k(H)=O_{k,\Delta}(n)$ for any $n$-vertex graph $H$ with maximum degree $\Delta$. Finally, building on many prior developments (e.g.\ \cite{FoSu, KoSu}), the full Burr--Erd\H os conjecture was proved by Lee \cite{Lee} in 2017.

There are many analogous questions and results for directed graphs (henceforth \emph{digraphs}). We assume all digraphs are simple and oriented, so they do not contain self-loops, parallel edges or anti-parallel edges. For a digraph $H$, define the
$k$-color \emph{oriented Ramsey number} $\ovar{r_{k}}(H)$ to be the minimum
$N$ such that any $k$-edge-colored tournament on $N$ vertices contains
a monochromatic copy of $H$. Note that if $H$ contains a directed
cycle, then $H$ does not appear in any transitive tournament, so
$\ovar{r_{k}}(H)$ only exists for acyclic $H$. Henceforth,
we work exclusively with acyclic digraphs $H$.

Unlike undirected Ramsey numbers, oriented Ramsey
numbers are interesting even in the case of one color, $k=1$, where
$\ovar{r_{1}}(H)$ is simply the minimum $N$ such that any tournament on
$N$ vertices contains a copy of $H$. Let $\TT_{n}$ denote the transitive
tournament on $n$ vertices. The study of oriented Ramsey numbers
was initiated by Stearns \cite{Stearns} in 1959 and Erd\H os and
Moser \cite{ErMo} in 1964, who showed the upper and lower bounds,
respectively, in 
\begin{equation}
2^{n/2-1}\le \ovar{r_{1}}(\TT_{n})\le2^{n-1}.\label{eq:transitive-tournament}
\end{equation}
The exponential constants in these bounds have not been improved, similar to the classical case of the diagonal undirected Ramsey number $r_2(K_n)$ \cite{Conlon,Sah,Spencer}. This may not be surprising, given that $\ovar{r_{1}}(\TT_{n})\le r_2(K_{n})$. Thus, improving the lower bound in (\ref{eq:transitive-tournament})
is at least as difficult as improving the lower bound on $r_2(K_{n})$.

Somewhat more is known about the oriented Ramsey number $\ovar{r_{k}}(H)$
when $H$ is sparse. When $H=P_{n}$ is the directed path on
$n$ vertices, Chv\'atal \cite{Chvatal} and Gy\'arf\'as and Lehel
\cite{GyLe} determined that $\ovar{r_{k}}(P_{n})=(n-1)^{k}+1$ using
the Gallai--Hasse--Roy--Vitaver theorem \cite{Gallai, Hasse, Roy, Vitaver}. In the case of one color, it
was more generally conjectured by Sumner in 1971 that for any oriented tree $T$
on $n\ge2$ vertices, $\ovar{r_{1}}(T)\le2n-2$. Sumner's conjecture has
received a considerable amount of attention over the years (see e.g.\
\cite{DrHa,ElSahili,HagTh,Havet,HavTh,KuMyOs,Thomason}); it was proven for $n$ sufficiently large by K\"uhn,
Mycroft, and Osthus \cite{KuMyOs}, and Dross and Havet \cite{DrHa} showed that
$\ovar{r_{1}}(T)\le \frac{21}{8}n-\frac{47}{16}$ for all $n\ge2$. In more colors, it was shown by
Buci\'c, Letzter, and Sudakov \cite{BuLeSu} that $\ovar{r_{k}}(T)= O_{k}(|V(T)|^{k})$
for any oriented tree $T$ and all $k\ge1$. In the same paper, they
asked a natural directed analogue of the classical Burr--Erd\H os
problem. 
\begin{problem}[\cite{BuLeSu}]
\label{prob:directed-burr-erdos} Is it true that $\ovar{r_{1}}(H)=O_{\Delta}(n)$
for every acyclic digraph $H$ with $n$ vertices and maximum degree
$\Delta$?
\end{problem}

Here we write $N^{+}(v)$ for the
out-neighborhood and $N^{-}(v)$ for the in-neighborhood of a vertex
$v\in V(H)$, and say that a digraph $H$ has maximum degree $\Delta$ if $\max_{v\in V(H)}(|N^{+}(v)|+ |N^{-}(v)|) = \Delta$.

Yuster \cite{Yuster} recently initiated the study of the special
case of Problem \ref{prob:directed-burr-erdos} when $H=P_{n}^{\ell}$ is the
$\ell$-th power of a directed path $P_{n}$, which is the digraph on vertex set $[n]$ whose edges are the ordered pairs $(i,j)$ satisfying $1\le j-i \le \ell$.
 This case was recently settled by Dragani\'c et al.\ \cite{Draganicetal}, who showed that $\ovar{r_{1}}(P_{n}^{\ell})=O_{\ell}(n)$.
Letting the \emph{bandwidth} of an acyclic digraph $H$ on $n$ vertices be the minimum $\ell$
such that $P_{n}^{\ell}$ contains a copy of $H$, this aforementioned result
implies that $\ovar{r_{1}}(H)=O_{\ell}(n)$ if $H$ has $n$ vertices and bandwidth $\ell$.

In this paper, we answer Problem \ref{prob:directed-burr-erdos} in
the negative, and show that in fact $\ovar{r_{1}}(H)$ can grow faster than
any fixed power of $n$, as long as the maximum degree is a sufficiently large constant.
\begin{thm}
\label{thm:general-lower-bound}For any $\Delta\ge2$ and $n$ sufficiently large in terms of $\Delta$, there exists
an acyclic digraph $H$ on
$n$ vertices and maximum degree at most $\Delta$ for which
\[
\ovar{r_{1}}(H)\geq n^{\Omega(\Delta^{2/3}/\log^{5/3}\Delta)}.
\]
\end{thm}

The power $\Delta^{2/3}$ seems to be best possible with our method,
but the power of $\log\Delta$ can be improved. Although the answer to Problem \ref{prob:directed-burr-erdos} is negative, we prove an almost linear upper bound on $\ovar{r_1}(H)$ for almost all $H$, in the following sense. Define $\overrightarrow{G}(n,d)$
to be the orientation of the random regular graph $G(n,d)$ on vertex
set $[n]$ with all edges pointing to the right.
\begin{thm}
\label{thm:G(n,d)-almost-linear}If $d\ge2$ is fixed and $H=\overrightarrow{G}(n,d)$,
then w.h.p.\footnote{As usual, we say that an event $\mathcal E$ happens \emph{with high probability (w.h.p.)}\ if $\Pr(\mathcal E) \to 1$ as $n\to \infty$.}\ (as $n\rightarrow\infty$)
\[
\ovar{r_{1}}(H)\leq n(\log n)^{4\log d}.
\]
\end{thm}

It is not difficult to extend Theorem \ref{thm:G(n,d)-almost-linear} to
show $\ovar{r_{1}}(H)=n(\log n)^{O_\Delta (1)}$ w.h.p.\ if $H$ is the forward acyclic orientation
of a uniformly random graph with any fixed degree sequence $\Delta=d_{1}\ge d_{2}\ge\cdots\ge d_{n}$, and therefore also for the forward acyclic orientation of a uniformly random bounded-degree graph. We also prove a similar bound $\ovar{r_1}(H)\le n(\log n)^{O_d(1)}$ when $H = \ovar G(n,p)$ is the forward acyclic orientation of an Erd\H os--R\'enyi random graph of constant average degree $d =pn$.

Although we are able to show that $\ovar{r_{1}}(H)$ is w.h.p.\ almost linear for a random bounded-degree acyclic digraph, we have not determined the worst-case behavior of this Ramsey number. For general acyclic digraphs $H$ on $n$ vertices with maximum degree $\Delta$, the best upper bound we are able to prove is $\ovar{r_1}(H) \leq n^{O_\Delta(\log n)}$ (see Theorem~\ref{thm:multicolor-upper-bound} below). Nonetheless, we are able to prove stronger (and in some cases linear) upper bounds on $\ovar{r_1}(H)$ in case $H$ lies in certain natural families. We now give two examples of such families.

Let the \emph{height} of $H$ be the number of vertices on the longest directed path in $H$. Equivalently, the height can be seen as a directed analogue of the chromatic number: $H$ has height at most $h$ if and only if $V(H)$ can be partitioned into independent sets $S_1,\dotsc,S_h$ such that every edge between $S_i$ and $S_j$ is oriented from $S_i$ to $S_j$, for every $i<j$. For acyclic digraphs of bounded height and bounded degree, we are able to prove the following linear bound on $\ovar{r_{1}}(H)$.

\begin{thm}
\label{thm:height-ub}If $H$ is an acyclic digraph on $n$ vertices
with maximum degree $\Delta$ and height $h$, then
\[
\ovar{r_{1}}(H)\le(\Delta h)^{10\Delta\log h}n.
\]
In particular, $\ovar{r_1}(H) = O_{h, \Delta}(n)$.
\end{thm}
Note that this theorem also implies the aforementioned $n^{O_\Delta(\log n)}$ upper bound on $\ovar{r_1}(H)$ for any bounded-degree acyclic digraph $H$, since the height of an acyclic digraph is at most its vertex count.

Next, we say that an acyclic digraph $H$ of height $h$ is \emph{graded}
if its vertex set can be partitioned into $h$ independent sets $S_{1},\dotsc,S_{h}$
such that every edge in $H$
is directed from some $S_{i}$ to $S_{i+1}$. Equivalently, $H$ is graded if for every pair of vertices $(u,v)$, all directed paths from $u$ to $v$ have the same length (the equivalence of the definitions follows e.g.\ from \cite[Proposition 4.4]{Lovasz}). A natural example
of a graded digraph is a grid (in any dimension) with all edges
oriented towards the first orthant. In general, a graded digraph can be obtained from any graded lattice (in the sense of partially ordered sets) $L$ by orienting every edge $x<y$ of the Hasse diagram of $L$ from $x$ to $y$. For graded digraphs of bounded degree, we are able to prove a polynomial bound on $\ovar{r_1}(H)$.
\begin{thm}\label{thm:graded}
If $H$ is a graded digraph on $n$ vertices with maximum degree $\Delta$ and height $h$, then
\[
\ovar{r_1}(H) \leq h^{10\Delta \log \Delta} n.
\]
In particular, since $h \leq n$, we have that $\ovar{r_1}(H)\leq n^{O(\Delta \log \Delta)}$.
\end{thm}

Our methods are motivated by those used by Conlox, Fox,
Lee, and Sudakov \cite{CoFoLeSu} to prove bounds on ordered Ramsey
numbers, and the two problems are especially closely related when the number of colors is at least $2$. Using this connection, we are able to give a super-polynomial
lower bound for $\ovar{r_{k}}(H)$ when $k\ge2$.
\begin{thm}
\label{thm:multicolor-lower-bound} For any $n\ge 4$, there exists an acyclic digraph $H$ on $n$ vertices with maximum degree
$3$ for which
\[
\ovar{r_{k}}(H)\ge n^{\log n/20\log\log n}
\]
for all $k \geq 2$.
\end{thm}
Thus, for acyclic digraphs of bounded degree, $\ovar{r_k}(H)$ can grow super-polynomially if $k\ge 2$.
In the other direction, for any number of colors we have the following
quasi-polynomial upper bound.
\begin{thm}
\label{thm:multicolor-upper-bound}If $k\ge1$ and $H$ is any acyclic
digraph with $n$ vertices and maximum degree $\Delta$, then
\[
\ovar{r_{k}}(H)\le2^{O_{k,\Delta}\left((\log n)^{2^{2k-1}}\right)}.
\]
\end{thm}

For one color, there is still a gap between the polynomial lower bound and the super-polynomial upper bound.

We remark that there is another well-studied analogue of ordinary Ramsey numbers in the directed setting, namely the \emph{directed Ramsey number} $\overleftrightarrow{r_k}(H)$, introduced by Bermond \cite{Bermond}. This is defined as the minimum $N$ such that a monochromatic copy of $H$ exists in every $k$-coloring of the edges of $\overleftrightarrow{K_N}$, the digraph with edges in both directions between all pairs of vertices. There are similarities and differences between the two theories (see e.g.\ \cite{BuLeSu}), and several of our results on oriented Ramsey numbers can be extended to the setting of directed Ramsey numbers. We expand on these connections in our concluding remarks, Section \ref{sec:concluding}.

To conclude this introduction, we remark on an interesting phenomenon brought to light by our results. An important ``meta-question'' driving many advances in Ramsey theory asks which graph parameters roughly determine a Ramsey number. In the Ramsey theory of undirected graphs, this question has been more or less resolved, at least in a qualitative sense.  Namely, for an undirected graph $H$, the degeneracy and the number of vertices are the main parameters that determine the growth of $r_2(H)$ (we focus on the two-color case to be concrete). This is easiest to see in the lower bounds: if $H$ has $n$ vertices, then certainly $r_2(H)\geq n$. Additionally, if $H$ has degeneracy $d$, then it is a simple exercise to show that a random $2$-edge-coloring on $2^{d/2}$ vertices does not contain a monochromatic copy of $H$ w.h.p.,\ implying that $r_2(H)\geq 2^{d/2}$. Putting these two facts together, we find that
\[
    \log r_2(H) =\Omega(d +\log n).
\]
Conlon, Fox, and Sudakov \cite[Conjecture 2.16]{CoFoSu} conjectured that this bound is tight up to the implicit constant, namely that
\[
    \log r_2(H) = \Theta(d + \log n)
\]
for any graph $H$ with $n$ vertices and degeneracy $d$. Moreover, this conjecture is known to be true up to a multiplicative factor of $\log^2 d$ \cite[Theorem 3.1]{FoSu}. Because of these results, one can say that the degeneracy and the vertex count of $H$ roughly determine its Ramsey number.

For acyclic digraphs, we do not know what parameters determine the growth order of $\ovar{r_1}(H)$ (indeed, we don't even know if $\ovar{r_1}(H)$ is polynomial or super-polynomial in $n$ when $H$ has bounded degree). Nonetheless, our results indicate that one parameter, which we call ``multiscale complexity,'' is relevant. Namely, suppose we order the vertices of $H$ as $v_1,\dotsc,v_n$ so that every edge is oriented to the right, that is $v_i \to v_j$ is an edge only if $i<j$ (such an ordering is often called a \emph{topological sorting} of $H$). Under this ordering we may assign every edge $v_i \to v_j$ a \emph{length} $j-i$. Here, if $u,v\in V(H)$, we write $u\rightarrow v$
to signify that there is a directed edge from $u$ to $v$, and similarly
$u\leftarrow v$ for an edge in the other direction. 

In graphs of bounded bandwidth, every edge is short and has length $O(1)$. At the other extreme, if $H$ has bounded height, then most edges of $H$ are long and have length $\Omega(n)$. The same is true in the random model $\ovar G(n,d)$, where most edges have length $\Omega(n)$. Some other acyclic digraphs have intermediate edge-length statistics, such as the \emph{directed grid} whose vertex set is $[\sqrt{n}]^2$ and all edges are oriented towards the lexicographically larger ordered pair. For every acyclic ordering of such a grid, there are many edges in most dyadic length scales between $1$ and $\Omega(\sqrt n)$. 

Loosely, let us say that $H$ has \emph{high multiscale complexity} if, for most dyadic intervals $I_t=[2^t, 2^{t+1})$ with $0 \leq t \leq \log n$, there are many edges in $H$ whose length is in $I_t$. Conversely, if most edge lengths of $H$ are concentrated in a small number of dyadic intervals, then we loosely say that $H$ has \emph{low multiscale complexity}. At a high level, all of our upper bound results prove upper bounds on $\ovar{r_1}(H)$ in terms of $n, \Delta$, and the multiscale complexity of $H$; if the multiscale complexity is low, then these bounds are stronger, and one can prove linear, nearly linear, or polynomial bounds on $\ovar{r_1}(H)$ (depending on the precise context). Conversely, our lower bound construction for Theorem \ref{thm:general-lower-bound} is a family of digraphs which we call \emph{interval meshes}. We delay the precise definition to Section \ref{sec:Lower-bounds}, but interval meshes are in some sense designed to maximize multiscale complexity among all graphs of maximum degree at most $\Delta$. 

We stress that we have not fully solved the problem of which natural parameters determine the growth order of $\ovar{r_1}(H)$, and we think this problem deserves further research. Nonetheless, our results do make it clear that some notion of multiscale complexity is one of these parameters, and we believe this notion is fundamental. As such, we state and prove many of our technical upper bounds in greater generality than is needed to deduce our main theorems, in order to demonstrate how notions of multiscale complexity arise naturally from our techniques. 

The rest of the paper is laid out as follows. Section~\ref{sec:Lower-bounds} carries out the construction of interval meshes to prove the lower bound Theorem~\ref{thm:general-lower-bound}. Section~\ref{sec:Greedy-embedding} uses the greedy embedding technique to prove the main technical lemmas needed for the upper bounds Theorems~\ref{thm:G(n,d)-almost-linear},~\ref{thm:height-ub}, and~\ref{thm:graded}. Section~\ref{sec:upper-bounds} completes the proofs of these results, as well as a more general upper bound in terms of ``multiscale complexity.''  Using the connection to ordered Ramsey numbers, Section~\ref{sec:multicolor&ordered} proves Theorems~\ref{thm:multicolor-lower-bound} and~\ref{thm:multicolor-upper-bound}. Finally, in our concluding remarks, Section~\ref{sec:concluding}, we collect a few of the tantalizing open problems that remain in this area.

\paragraph{Notation and terminology.}
  By an \emph{embedding} $H \hookrightarrow T$, we mean an injective function $V(H) \to V(T)$ such that edges of $H$ are mapped to edges of $T$, with edge orientations preserved. We say that a digraph $T$ is $H$-free if there is no embedding $H\hookrightarrow T$. All logarithms are to base $2$.  For the sake of clarity of
presentation, we sometimes omit floor and ceiling signs when they are not crucial.

\section{Lower bounds \label{sec:Lower-bounds}}

In this section, we prove the lower bound Theorem \ref{thm:general-lower-bound},
which states that for any $\Delta\ge1$ there exists a family of acyclic
digraphs $\{H_{n}\}_{n\ge1}$ with maximum degree $\Delta$ for which
$|V(H_{n})|=n$ and $\ovar{r_{1}}(H_{n})\ge n^{\Omega(\Delta^{2/3}/\log^{5/3}\Delta)}$.
The lower bound consists of three ingredients: constructing a bounded
degree acyclic digraph $H$ that is hard to embed, constructing a
Ramsey tournament $T$ that is hard to embed $H$ into, and proving that
there is no embedding $H\hookrightarrow T$.

The next three subsections separately provide these three ingredients.
Section \ref{subsec:Interval-meshes} defines a class of digraphs
$H$ with edges ``at all scales,'' which we call \emph{interval meshes,}
and proves the existence of bounded-degree $H$ with this property.
Section \ref{subsec:Walks-in-tournaments} defines the Ramsey tournament
$T$ as a lexicographic power of a tournament $R$ with no large transitive
subtournament, and shows that embeddings $H\hookrightarrow T$ correspond
to certain highly constrained walks on $R$ which we call \emph{$(R,f,s)$-walks.}
Section \ref{subsec:Completing-the-proof} completes the proof by
showing that long $(R,f,s)$-walks do not exist, and therefore large
interval meshes $H$ cannot be embedded into small powers $T=R^{m}$. 

\subsection{Interval meshes \label{subsec:Interval-meshes}}

Our proof of Theorem \ref{thm:general-lower-bound} is motivated by
the lower bound construction for ordered Ramsey numbers proved by
Conlon, Fox, Lee, and Sudakov \cite{CoFoLeSu}. They prove a lower bound on the ordered Ramsey number of a random matching on $[n]$; see Theorem \ref{thm:scrambled-matchings} and the surrounding discussion for details. The main
property they need of the random matching is that most pairs of
long intervals have an edge between them. We need the following stronger
property of the same form for our acyclic digraph $H$.
\begin{defn}
If $f:\mathbb{N}\rightarrow\mathbb{R}_{>0}$ is a nondecreasing function,
we define an \emph{$f$-interval mesh }to be an acyclic digraph $H$
whose vertex set is an interval $I\subseteq\mathbb{N}$ such that
any edge $(i,j)\in E(H)$ has $i<j$ and for all pairs of intervals
$(a_{1},b_{1}],(a_{2},b_{2}]\subseteq I$ with $b_{1}\le a_{2}$ and
\begin{equation}
a_{2}-b_{1}\le f(\min(b_{1}-a_{1},b_{2}-a_{2})),\label{eq:distance-1}
\end{equation}
there exists an edge in $H$ between $(a_{1},b_{1}]$ and $(a_{2},b_{2}]$.
When the function $f$ is clear from context, we simply call $H$
an \emph{interval mesh.}
\end{defn}
The notion of an interval mesh is one way of formalizing the notion from the introduction of ``high multiscale complexity'', since interval meshes have many edges in every length scale.
We construct interval meshes of bounded degree using a greedy algorithm. 
\begin{lem}
\label{lem:interval-mesh}For any nondecreasing function $f:\mathbb{N}\rightarrow\mathbb{R}_{>0}$
with $S=\sum_{m\ge0}f(2^{m+2})\cdot2^{-2m}<\infty$, there exists
an $f$-interval mesh $H$ on vertex set $\mathbb{N}$ with maximum degree
at most $2S+17$. 
\end{lem}

\begin{proof}
Starting
from an empty digraph on $\mathbb{N}$, we construct $H$ by using a greedy algorithm to add edges. All edges introduced
point to the right, so the resulting digraph must be acyclic. Define
$I_{m,j}$ to be the dyadic interval $(j\cdot2^{m},(j+1)\cdot2^{m}]$
with length $2^{m}$.

Let $m\ge0$ range through the nonnegative integers. On subroutine
$m$, we iterate through all pairs $(i,j)\in\mathbb{N}_{\ge0}^{2}$
satisfying 
\begin{equation}
1\le j-i\le f(2^{m+2})\cdot2^{-m}+4\label{eq:dyadic-distance-1}
\end{equation}
and add an edge between the (currently) lowest degree vertex of $I_{m,i}$
and the (currently) lowest degree vertex of $I_{m,j}$, if an edge
does not exist between $I_{m,i}$ and $I_{m,j}$ already. Writing
$d(U)$ for the sum of the degrees of the vertices in a set $U$,
we see that for any given $i$, subroutine $m$ increases $d(I_{m,i})$
by a total of at most 
\[
2[f(2^{m+2})\cdot2^{-m}+4]=2^{-m+1}\cdot f(2^{m+2})+8.
\]

Let $H$ be the digraph produced from this infinite process. Explicitly,
$H$ is the edge union of all the graphs $H^{(m)}$, where $H^{(m)}$ is
the graph produced after subroutine $m$. It has the property that
every pair of dyadic intervals $I_{m,i}$, $I_{m,j}$ satisfying (\ref{eq:dyadic-distance-1})
has an edge between them.

We first check that $H$ has maximum degree at most $2S+17$. Since
$I_{m,i}$ is a union of $2^{m-k}$ dyadic intervals of length $2^{k}$,
we have that after subroutine $m$,
\[
d(I_{m,i})\le\sum_{k=0}^{m}2^{m-k}(2^{-k+1}\cdot f(2^{k+2})+8)\le2^{m+1}\cdot S+2^{m+1}\cdot 8=2^{m}(2S+16).
\]

In particular, the average degree in $I_{m,i}$ is less than $2S+17$
at the end of subroutine $m$. However, subroutine $m$ only adds
edges incident to vertices of $I_{m,i}$ which have at most average
degree, so no new vertex of degree at least $2S+18$ can appear on
subroutine $m$. Thus, no vertex of degree at least $2S+18$ ever appears,
as desired.

Next, we check that $H$ is an $f$-interval mesh. Suppose two intervals
$(a_{1},b_{1}]$, $(a_{2},b_{2}]$ with $b_{1}\le a_{2}$ satisfy
(\ref{eq:distance-1}) and let $m$ be the largest positive integer
such that $(a_{1},b_{1}]$ and $(a_{2},b_{2}]$ both contain dyadic
intervals of length $2^{m}$. Note that in any interval $(a,b]$ of
integers, a longest dyadic subinterval $I_{\ell,i}$ has length $2^{\ell}\in[\frac{b-a}{4},b-a]$,
so in particular $2^{m+2}\ge\min(b_{1}-a_{1},b_{2}-a_{2}).$ Next,
pick $i$ maximal and $j$ minimal such that $I_{m,i}\subseteq(a_{1},b_{1}]$
and $I_{m,j}\subseteq(a_{2},b_{2}]$.

Using (\ref{eq:distance-1}), we find
\[
j-i\le4+\frac{a_{2}-b_{1}}{2^{m}}\le4+\frac{f(\min(b_{1}-a_{1},b_{2}-a_{2}))}{2^{m}}\le4+\frac{f(2^{m+2})}{2^{m}},
\]
and so there is an edge in $H$ between $I_{m,i}$ and $I_{m,j}$,
and therefore between $(a_{1},b_{1}]$ and $(a_{2},b_{2}]$ as well.
\end{proof}
The acyclic digraphs $H_{n}$ we use are induced subgraphs on intervals of
the infinite interval mesh constructed in Lemma \ref{lem:interval-mesh},
for an appropriate choice of $f$.

\subsection{Walks in tournaments \label{subsec:Walks-in-tournaments}}

Next, we construct the large tournament $T$ which is difficult to
embed $H$ into. Again, the construction is motivated by the lower bound argument
of Conlon, Fox, Lee, and Sudakov \cite{CoFoLeSu} for ordered Ramsey
numbers, although its analysis requires new techniques.

Recall that the \emph{lexicographic product} $G\cdot H$ of two digraphs
$G$ and $H$ is the digraph on vertex set $V(G)\times V(H)$ with
an edge $(g_{1},h_{1})\rightarrow(g_{2},h_{2})$ iff $g_{1}\rightarrow g_{2}$
in $G$ or $g_{1}=g_{2}$ and $h_{1}\rightarrow h_{2}$ in $H$. Henceforth,
write $G^{m}$ for the $m$-fold lexicographic product of $G$ with
itself. Note that lexicographic powers of tournaments are tournaments. By (\ref{eq:transitive-tournament}),
there exists for any $r\ge3$ a tournament $R$ on $r$
vertices with no transitive subtournament of size $2\log r+2$.
 We
take $T=R^{m}$ and show that an interval mesh $H$ is difficult
to embed into $T$. 

To this end, let $H$ be an interval mesh. We relate embeddings $H\hookrightarrow R^{m}$
to certain constrained walks on the tournament $R$.
\begin{defn}
For a tournament $R$, a nondecreasing function $f:\mathbb{N}\rightarrow\mathbb{R}_{>0}$,
and $s\ge1$, define an $(R,f,s)$\emph{-walk }to be a sequence of
ordered pairs $\{(v_{i},a_{i})\}_{i=1}^{\ell}$ $(\ell$ possibly
infinite) where for each $i$, $v_{i}\in V(R)$, $1\le a_{i}\le s$,
$v_{i}\ne v_{i+1}$, and for any pair $i<j$ for which $v_{i}\leftarrow v_{j}$
in $R$, we have
\[
a_{(i,j)}>f(\min(a_{i},a_{j})),
\]
where $a_{I}\coloneqq\sum_{k\in I}a_{k}$ if $I$ is an interval of
integers and the empty sum equals $0$. We define the \emph{length}
of an $(R,f,s)$-walk to be $a_{[1,\ell]}$.
\end{defn}

Let $L_{R,f}(s)$ be the length of the longest $(R,f,s)$-walk if
such a maximum exists, and $+\infty$ otherwise. The next lemma reduces
Theorem \ref{thm:general-lower-bound} to showing asymptotic upper
bounds on $L_{R,f}(s)$.
\begin{lem}
\label{lem:walk-reduction}Suppose there exists $r\ge1$, a nondecreasing
$f:\mathbb{N}\rightarrow\mathbb{R}_{>0}$, and a tournament $R$ on
$r$ vertices for which $\limsup_{s\rightarrow\infty}L_{R,f}(s)s^{-1}=\alpha$.
If $H$ is an $f$-interval mesh on $n$ vertices, then
\[
\ovar{r_{1}}(H)\ge n^{\frac{\log r}{\log\alpha}-o(1)}.
\]
\end{lem}

\begin{proof}
For each $n\ge1$, let $m=m(n)$ be the minimum positive integer for
which there exists an $f$-interval mesh $H$ with vertex set $[n]$ and
an embedding $\phi:H\hookrightarrow R^{m}$.

Let $\pi:R^{m}\rightarrow R$ be the projection onto the first coordinate.
Consider the sequence $\{\pi\circ\phi(j)\}_{j=1}^{n}$. Consecutive terms of
this sequence may repeat, so we block the sequence into consecutive
runs of identical vertices. Suppose there are $\ell$ total runs $I_{1}\sqcup\cdots\sqcup I_{\ell}=[n]$
and run $I_{i}$ consists of $a_{i}$ repetitions of vertex $v_{i}\in V(R)$.
We claim that $\{(v_{i},a_{i})\}_{i=1}^{\ell}$ is an $(R,f,s)$-walk,
where $s\coloneqq\max(a_{i})$. It is easy to see that $1\leq a_{i}\le s$
for all $i$, and that $v_{i}\ne v_{i+1}$ since we already blocked
consecutive identical vertices together. It remains to check the key
condition, that if $i<j$ and $v_{i}\leftarrow v_{j}$ in $R$, we
must have
\[
a_{(i,j)}>f(\min(a_{i},a_{j})).
\]
Suppose this is not true, so there exists some $i<j$ with $v_{i}\leftarrow v_{j}$
and $a_{(i,j)}\le f(\min(a_{i},a_{j}))$. By the definition of $v_i,v_j$, we have that $\pi(\phi(I_{i}))=\{v_{i}\}$,
$\pi(\phi(I_{j}))=\{v_{j}\}$. By the definition of the lexicographic
power, if $v_{i}\leftarrow v_{j}$ then $w_{i}\leftarrow w_{j}$ for
all $w_{i}\in\pi^{-1}(v_{i}),w_{j}\in\pi^{-1}(v_{j})$. Thus, all
edges between $\phi(I_{i})$ and $\phi(I_{j})$ are oriented from $\phi(I_{j})$
to $\phi(I_{i})$. For $\phi$ to be a homomorphism, this means that
$H$ has no edge oriented from $I_{i}$ to $I_{j}$. On the other
hand, $|I_{i}|=a_{i}$, $|I_{j}|=a_{j}$, and the distance between
the two intervals is $a_{(i,j)}\le f(\min(a_{i},a_{j}))$, so since
$H$ is an $f$-interval mesh there is such an edge. This is a contradiction,
and we see that $\{(v_{i},a_{i})\}_{i=1}^{\ell}$ is an $(R,f,s)$-walk
of length $n$.

Fix any $\varepsilon>0$. We are given that $L_{R,f}(s)\le(\alpha+\varepsilon)s$
for sufficiently large $s$, so we get $s\ge(\alpha+\varepsilon)^{-1}n$ using the fact that $n \leq L_{R,f}(s)$ since we have found an $(R,f,s)$-walk of length $n$.
Since $s=\max(a_{i})$, this means that there is some subinterval
$I_{i}\subseteq[n]$ of length at least $(\alpha+\varepsilon)^{-1}n$
for which $\pi\circ\phi$ is constant on $I_{i}$, i.e.\ the image
$\phi(I_{i})\subseteq \{v_{i}\}\times R^{m-1}$ lies in a copy of $R^{m-1}$. Putting
this together, we have shown that for large enough $n$, the existence
of an embedding $\phi:H\hookrightarrow R^{m}$ implies the existence
of an embedding $\phi':H'\hookrightarrow R^{m-1}$ for some $f$-interval
mesh $H'$ on $(\alpha+\varepsilon)^{-1}n$ vertices. In other words,

\[
m(n)\ge m((\alpha+\varepsilon)^{-1}n)+1
\]
for all $n$ sufficiently large. Applying this recursively, we obtain
that $m(n)\ge\frac{\log n}{\log(\alpha+\varepsilon)}-O_\varepsilon(1)$. By the definition of $m(n)$, any $f$-interval
mesh $H$ on $n$ vertices has no embedding into $R^{m(n)-1}$, and
so
\[
\ovar{r_{1}}(H)>|V(R^{m(n)-1})|\ge r^{\frac{\log n}{\log(\alpha+\varepsilon)}-O_\varepsilon(1)}=\Omega_{\varepsilon}\left(n^{\frac{\log r}{\log(\alpha+\varepsilon)}}\right),
\]
which proves the lemma.
\end{proof}
To finish the proof of Theorem \ref{thm:general-lower-bound}, it
remains to bound the asymptotic growth rate of $L_{R,f}(s)$, which
we do in the next section.

\subsection{Completing the proof \label{subsec:Completing-the-proof}}

The next lemma is a simple observation that is helpful for analyzing
the structure of $(R,f,s)$-walks.
\begin{lem}
\label{lem:t-sequence}If $R$ is a tournament without a transitive
$t$-subtournament, then any sequence $v_{1},\ldots,v_{t}$ of $t$
vertices either contains a back-edge $v_{i}\leftarrow v_{j}$ with
$i<j$ or two consecutive elements $v_{i}=v_{i+1}$.
\end{lem}

\begin{proof}
If the vertices are all distinct, then since $R$ has no transitive
$t$-subtournament there must exist a back-edge. Suppose $v_{i}=v_{j}$
and $i<j$ but $j\ne i+1$. Then either $v_{i}\leftarrow v_{i+1}$
or $v_{i+1}\leftarrow v_{j}$ is a back-edge.
\end{proof}
We now prove a recursive upper bound on $L_{R,f}(s)$. Given
an implicit parameter $t\ge3$, tournament $R$, and nondecreasing function
$f:\mathbb{N}\rightarrow\mathbb{R}_{>0}$, we say that a positive
integer $s$ is \emph{short} if $L_{R,f}(s)\le2st$ and $L_{R,f}(s'')\le f(s'')$
for all $s''\le s$.
\begin{lem}
\label{lem:Rfs}Suppose $s\ge1$, $t\ge3$, $R$ is a tournament without
a transitive $t$-subtournament, and $f:\mathbb{N}\rightarrow\mathbb{R}_{>0}$
is a nondecreasing function satisfying $f(s)>6st^{2}$. If $s$ is
short, then every $s'\in[2st,4st]$ is short.
\end{lem}

\begin{proof}
Suppose $s<s'\le4st$, $\{(v_{i},a_{i})\}_{i=1}^{\ell}$ is an $(R,f,s')$-walk,
and let $i_{1}<\dotsb<i_{u}$ be the sequence of all indices where $a_{i_{j}}\ge s$. 

Our first goal is to show that $u<t$. If not, by Lemma \ref{lem:t-sequence}
there is either a back-edge $v_{i_{x}}\leftarrow v_{i_{y}}$ with
$x<y\le t$ or some $x\le t-1$ for which $v_{i_{x}}=v_{i_{x+1}}$.
We show that neither of these situations is possible.

In the first case, there is an edge $v_{i_{x}}\leftarrow v_{i_{y}}$
with $x<y\le t$. By the definition of an $(R,f,s')$-walk,
\begin{equation}
a_{(i_{x},i_{y})}>f(\min(a_{i_{x}},a_{i_{y}}))\ge f(s).\label{eq:backedge-length}
\end{equation}
On the other hand, 
\[
a_{(i_{x},i_{y})}=\sum_{j\in(x,y)}a_{i_{j}}+\sum_{j\in[x,y)}a_{(i_{j},i_{j+1})}\le s't+t\cdot L_{R,f}(s)\le6st^{2},
\]
since for each $j=x,\ldots,y-1$, the subsequence $\{(v_{i},a_{i})\}_{i=i_{j}+1}^{i_{j+1}-1}$
is an $(R,f,s)$-walk with length at most $L_{R,f}(s)\le2st$. But
$f(s)>6st^{2}$, so this contradicts (\ref{eq:backedge-length}) and
the back-edge $v_{i_{x}}\leftarrow v_{i_{y}}$ cannot exist. 

Next, suppose for some $x\le t-1$ that $v_{i_{x}}=v_{i_{x+1}}$.
The subsequence $\{(v_{i},a_{i})\}_{i=i_{x}+1}^{i_{x+1}-1}$ is an
$(R,f,s'')$-walk where $s''$ is the maximum value of $a_{i}$ in
this subsequence. Pick some $z$ for which $a_{z}=s''$. Either $v_{i_{x}}\leftarrow v_{z}$
or $v_{z}\leftarrow v_{i_{x+1}}$ is a back-edge, and without loss
of generality assume it is the former. Applying the definition of
the $(R,f,s'')$-walk on the two indices $i_{x}$ and $z$,
\[
a_{(i_{x},z)}>f(\min(a_{i_{x}},a_{z}))=f(s'').
\]
On the other hand, this sum is bounded above by $a_{(i_{x},i_{x+1})}$.
We know $a_{(i_{x},i_{x+1})}\le L_{R,f}(s'')\le f(s'')$ because $\{(v_{i},a_{i})\}_{i=i_{x}+1}^{i_{x+1}-1}$
is an $(R,f,s'')$-walk and $s\ge s''$ is short, so we have another
contradiction. Thus $u<t$. 

We obtain
\begin{equation}
a_{[1,\ell]}\le\sum_{j=1}^{u}a_{i_{j}}+\sum_{j=0}^{u}a_{(i_{j},i_{j+1})}\le s't+t\cdot L_{R,f}(s)\le s't+2st^{2},\label{eq:Rfs-recursion}
\end{equation}
for all $s<s'\le4st$ and any $(R,f,s')$-walk $\{(v_{i},a_{i})\}_{i=1}^{\ell}$.
Here we let $i_{0}=0$ and $i_{u+1}=\ell+1$ for convenience.

Inequality (\ref{eq:Rfs-recursion}) implies that $L_{R,f}(s')\le2s't$
for all $s'\in[2st,4st]$. It also implies that $L_{R,f}(s'')\le6st^{2}<f(s)\le f(s'')$
for all $s''\in(s,4st]$. We have verified both conditions for $s'$
to be short for every $s'\in[2st,4st]$, as desired.
\end{proof}
It remains to pick a function $f$ for which Lemma \ref{lem:Rfs}
bootstraps successfully. Define
\begin{equation}
f(s)\coloneqq\begin{cases}
\frac{10s^{2}t^{3/2}\log t}{\log^{2}s} & s\ge4\\
40t^{3/2}\log t & s<4,
\end{cases}\label{eq:f-definition}
\end{equation}
where the values of $f(1),f(2),f(3)$ are chosen just to make $f$
nondecreasing. Recall that all logarithms are to base $2$.
\begin{lem}
\label{lem:Rfs-specific}If $t \ge 10^6$, $f:\mathbb{N}\rightarrow\mathbb{R}_{>0}$
is defined by (\ref{eq:f-definition}), and $R$ is a tournament without
a transitive $t$-subtournament, then $L_{R,f}(s)\le2st$ for all $s$ sufficiently large.
\end{lem}

\begin{proof}
We apply Lemma \ref{lem:Rfs} inductively to show that $s$ is
short for all $s$ sufficiently large. This implies the desired result.

For the base case, we claim that $s$ is short for all $s\le s_{0}\coloneqq40t^{1/2}\log t$.
Indeed, suppose $s\le s_{0}$ and $\{(v_{i},a_{i})\}_{i=1}^{\ell}$
is an $(R,f,s)$-walk. By Lemma \ref{lem:t-sequence}, if $\ell\ge t$
then there is either a back-edge $v_{i}\leftarrow v_{j}$ with $i<j\le t$
or two consecutive elements $v_{i}=v_{i+1}$. The latter contradicts
the definition of an $(R,f,s)$-walk, so assume the former holds.
But $a_{(i,j)}<st\le f(1)\le f(\min(a_{i},a_{j}))$, so this contradicts
the definition of an $(R,f,s)$-walk again. We have shown that $\ell<t$,
so $L_{R,f}(s)<st\le f(s)$ for all $s\le s_{0}$. This proves $s$
is short for every $s\le s_{0}$.

Next, we check that $f(s)>6st^{2}$ for all $s\ge s_{0}$. Indeed, $f(s)/s$ is increasing for $s\ge s_{0}$,
and $\log s_{0}\le\log t$ since $t \ge 10^6$. We get
\[
f(s_{0})=\frac{10s_{0}^{2}t^{3/2}\log t}{\log^{2}s_{0}}>1000t^{5/2}\log t>6s_{0}t^{2},
\]
proving that the conditions of Lemma \ref{lem:Rfs} are satisfied.
By induction, Lemma \ref{lem:Rfs} then implies that $s$ is short
for $s\in[(2t)^{k}s_{0},(4t)^{k}s_{0}]$ for all $k\ge0$. All sufficiently
large integers lie in some such interval, so $L_{R,f}(s)\le2st$ for
all sufficiently large $s$, as desired.
\end{proof}
Putting everything together, we have a proof of the general lower
bound.
\begin{proof}[Proof of Theorem \ref{thm:general-lower-bound}.]
We may assume that $\Delta$ is sufficiently large as we always have $\ovar{r_{1}}(H_{n})\ge n$, which proves the theorem for small $\Delta$ by picking the implicit constant factor in the exponent appropriately. Let $t = \Delta^{2/3}/(200 \log^{2/3} \Delta)$; we may assume $t\ge 10^6$. Define $f:\mathbb{N}\rightarrow\mathbb{R}_{>0}$
by (\ref{eq:f-definition}). We have
\[
S=\sum_{m\ge0}f(2^{m+2})\cdot2^{-2m} \le f(1)\cdot\sum_{m\ge1}m^{-2} \le 80t^{3/2}\log t.
\]
Lemma \ref{lem:interval-mesh} implies that there exists an $f$-interval
mesh $H$ on $\NN$ with maximum degree at most $2S+17 \le 200 t^{3/2}\log t \le \Delta$.
For any $n\ge 1$, let $H_n$ be the induced subgraph of $H$ on the interval $[n]$, so that $H_n$ has $n$ vertices and is also an $f$-interval mesh of maximum degree at most $\Delta$.

There exists a tournament $R$ on $r=2^{\Omega(t)}$ vertices with
no transitive $t$-subtournament. By Lemmas \ref{lem:walk-reduction}
and \ref{lem:Rfs-specific} applied with these choices of $R,f$ and
$H_{n}$, we find that 
\[
\ovar{r_{1}}(H_{n})>n^{\frac{\log r}{\log(2t)}-o(1)}\ge n^{\Omega(t/\log t)},
\]
which proves the theorem, by our choice of $t$.
\end{proof}

We remark that the polylogarithmic factor in $\Delta=\Omega(t^{2/3}/\log^{2/3} \Delta)$ can be easily improved. Indeed, the growth rate of $f(s)=\Theta(s^2/\log^2 s)$ is chosen so that
\[
\sum_{m\ge 0}f(2^{m+2})\cdot 2^{-2m} < \infty,
\]
and we may take $f(s)=\Theta(s^2/\log^{1+\varepsilon} s)$ for any fixed $\varepsilon>0$ instead, leading to a slightly smaller $\Delta$.

\section{Greedy embedding \label{sec:Greedy-embedding}}

In this section we prove the main lemmas needed for all of the upper
bounds in this paper. We use the greedy embedding technique, motivated
by similar arguments for ordered graphs from \cite{CoFoLeSu}.

\textbf{Framework.} Let $H$ be an acyclic digraph on $n$ vertices
$v_{1},\ldots,v_{n}$. We would like to find an embedding $\phi:H\hookrightarrow T$
into an ambient tournament $T$. In addition we specify $n$
sets $U_{1},\ldots,U_{n}\subseteq V(T)$ and aim to satisfy $\phi(v_{i})\in U_{i}$
for all $i$. Embedding then proceeds in $n$ rounds, where round
$t$ determines the image $\phi(v_{t})$. After round $t$, we keep
track of the shrinking sets $U_{1}^{(t)},\ldots,U_{n}^{(t)}$ of ``valid
candidates'' for each vertex, where initially $U_{i}^{(0)}=U_{i}$
for all $i$. On round $t$, $\phi(v_{t})$ is a carefully chosen vertex
in $U_{t}^{(t-1)}$, and $U_{t}^{(t)}\coloneqq\{\phi(v_{t})\}$. The
other candidate sets are updated as follows:
\[
U_{j}^{(t)}\coloneqq\begin{cases}
[U_{j}^{(t-1)}\cap N^{+}(v_{t})]\setminus\{\phi(v_{t})\} & \text{if }v_{t}\rightarrow v_{j},\\{}
[U_{j}^{(t-1)}\cap N^{-}(v_{t})]\setminus\{\phi(v_{t})\} & \text{if }v_{j}\rightarrow v_{t},\\
U_{j}^{(t-1)}\setminus\{\phi(v_{t})\} & \text{else}.
\end{cases}
\]

The process fails if there is an empty $U_{t}^{(t-1)}$, as there would be no valid choice for $\phi(v_t)$. Otherwise,
it succeeds if $\phi(v_{t})$ is chosen successfully for each $t\le n$.
Note that after round $t$, $\{\phi(v_{1}),\ldots,\phi(v_{t})\}$
is an embedded copy of $H[\{v_{1},\ldots,v_{t}\}]$ in $T$, and
these vertices have been removed from the other candidate sets, so
the update rule guarantees that $U_{j}^{(t)}=\{\phi(v_{j})\}$ remains
a singleton when $t\ge j$. If the greedy embedding process succeeds,
it exhibits the existence of a copy of $H$ in $T$. See Figure \ref{fig:greedy-embedding} for a schematic illustration of the greedy embedding process.

\begin{figure}[h]
    \centering
    \begin{tikzpicture}[
containernode/.style={rectangle, draw=black, ultra thin, inner sep=0pt, minimum height=25mm, minimum width=40mm},
roundnode/.style={circle, draw=black, fill=black, semithick, inner sep=0pt, minimum size=1mm},
roundnodeempty/.style={circle, draw=black, semithick, inner sep=0pt, minimum size=1mm},
vertexset/.style={rectangle, draw=black, semithick, minimum height=8mm, minimum width=11mm},
vertexsetsmall/.style={rectangle, draw=black, semithick, minimum height=5mm, minimum width=7mm},
vertexsettiny/.style={rectangle, draw=black, semithick, minimum height=2mm, minimum width=4mm},
]

\def\cx{0}
\def\ccx{5.5}
\def\cccx{11}
\def\xgap{1.25}
\def\ry{0}
\def\rry{-3}

\node[draw=black!0] at (1.3, 1.25) {Target graph $H$};
\node[containernode] (c) at (2, \ry + 0.25) {};
\node[roundnode][label={[label distance=-0.85cm]:$v_1$}] (n11) at (0.5, \ry) {};
\node[roundnode][label={[label distance=-0.85cm]:$v_2$}] (n12) at (1.5, \ry) {};
\node[roundnode][label={[label distance=-0.85cm]:$v_3$}] (n13) at (2.5, \ry) {};
\node[roundnode][label={[label distance=-0.85cm]:$v_4$}] (n14) at (3.5, \ry) {};
\draw[semithick, ->] (n11.north) .. controls +(up:4mm) and +(up:4mm) .. ([xshift=-.05cm]n12.north);
\draw[semithick, ->] (n11.north) .. controls +(up:6mm) and +(up:6mm) .. ([xshift=-.05cm]n13.north);
\draw[semithick, ->] (n12.north) .. controls +(up:6mm) and +(up:6mm) .. ([xshift=.05cm]n14.north);
\draw[semithick, ->] (n13.north) .. controls +(up:4mm) and +(up:4mm) .. ([xshift=-.05cm]n14.north);

\node[draw, ultra thin] at (5.5, 1.25) {Step 0};
\node[vertexset][rounded corners][label={[label distance=-1.5cm]:$U_1^{(0)}$}] (n21)   at (\ccx, \ry)         {};
\node[roundnodeempty][rounded corners]                                         (n21_c) at (\ccx, \ry)         {};
\node[vertexset][rounded corners][label={[label distance=-1.5cm]:$U_2^{(0)}$}] (n22)   at (\ccx+\xgap, \ry)   {};
\node[vertexsetsmall][dashed, rounded corners]                                 (n22_c) at (\ccx+\xgap, \ry)   {};
\node[vertexset][rounded corners][label={[label distance=-1.5cm]:$U_3^{(0)}$}] (n23)   at (\ccx+2*\xgap, \ry) {};
\node[vertexsetsmall][dashed, rounded corners]                                 (n23_c) at (\ccx+2*\xgap, \ry) {};
\node[vertexset][rounded corners][label={[label distance=-1.5cm]:$U_4^{(0)}$}] (n24)   at (\ccx+3*\xgap, \ry) {};
\draw[semithick, dashed, ->] (n21_c.north) .. controls +(up:4mm) and +(up:4mm) .. (n22_c.north);
\draw[semithick, dashed, ->] (n21_c.north) .. controls +(up:8mm) and +(up:8mm) .. (n23_c.north);

\node[draw, ultra thin] at (11, 1.25) {Step 1};
\node[roundnode][label={[label distance=-1.1cm]:$\phi(v_1)$}]                        (n31)   at (\cccx, \ry)         {};
\node[vertexsetsmall][rounded corners][label={[label distance=-1.35cm]:$U_2^{(1)}$}] (n32)   at (\cccx+\xgap, \ry)   {};
\node[roundnodeempty]                                                                (n32_c) at (\cccx+\xgap, 0)     {};
\node[vertexsetsmall][rounded corners][label={[label distance=-1.35cm]:$U_3^{(1)}$}] (n33)   at (\cccx+2*\xgap, \ry) {};
\node[vertexset][rounded corners][label={[label distance=-1.5cm]:$U_4^{(1)}$}]       (n34)   at (\cccx+3*\xgap, \ry) {};
\node[vertexsetsmall][rounded corners, dashed]                                       (n34_c) at (\cccx+3*\xgap, \ry) {};
\draw[semithick, ->] (n31.north) .. controls +(up:4mm) and +(up:4mm) .. (n32.north);
\draw[semithick, ->] (n31.north) .. controls +(up:6mm) and +(up:6mm) .. (n33.north);
\draw[semithick, dashed, ->] (n32_c.north) .. controls +(up:4mm) and +(up:4mm) .. (n34_c.north);

\node[draw, ultra thin] at (0.5, -1.75) {Step 2};
\node[roundnode][label={[label distance=-1cm]:$\phi(v_1)$}]                          (n41)   at (\cx, \rry)    {};
\node[roundnode][label={[label distance=-1cm]:$\phi(v_2)$}]                          (n42)   at (\cx + \xgap, \rry) {};
\node[vertexsetsmall][rounded corners][label={[label distance=-1.25cm]:$U_3^{(2)}$}] (n43)   at (\cx + 2*\xgap, \rry)  {};
\node[roundnodeempty]                                                                (n43_c) at (\cx + 2*\xgap, \rry)  {};
\node[vertexsetsmall][rounded corners][label={[label distance=-1.25cm]:$U_4^{(2)}$}] (n44)   at (\cx + 3*\xgap, \rry) {};
\node[vertexsettiny][rounded corners, dashed]                                        (n44_c) at (\cx + 3*\xgap, \rry) {};
\draw[semithick, ->] (n41.north) .. controls +(up:4mm) and +(up:4mm) .. ([xshift=-.05cm]n42.north);
\draw[semithick, ->] (n41.north) .. controls +(up:6mm) and +(up:6mm) .. ([xshift=-.05cm]n43.north);
\draw[semithick, ->] (n42.north) .. controls +(up:6mm) and +(up:6mm) .. ([xshift=1mm]n44.north);
\draw[semithick, dashed, ->] (n43_c.north) .. controls +(up:4mm) and +(up:4mm) .. (n44_c.north);

\node[draw, ultra thin] at (5.5, -1.75) {Step 3};
\node[roundnode][label={[label distance=-1cm]:$\phi(v_1)$}]                        (n51)   at (\ccx, \rry)         {};
\node[roundnode][label={[label distance=-1cm]:$\phi(v_2)$}]                        (n52)   at (\ccx+\xgap, \rry)   {};
\node[roundnode][label={[label distance=-1cm]:$\phi(v_3)$}]                        (n53)   at (\ccx+2*\xgap, \rry) {};
\node[vertexsettiny][rounded corners][label={[label distance=-1.1cm]:$U_4^{(3)}$}] (n54)   at (\ccx+3*\xgap, \rry) {};
\node[roundnodeempty]                                                              (n54_c) at (\ccx+3*\xgap, \rry) {};
\draw[semithick, ->] (n51.north) .. controls +(up:4mm) and +(up:4mm) .. ([xshift=-.05cm]n52.north);
\draw[semithick, ->] (n51.north) .. controls +(up:6mm) and +(up:6mm) .. ([xshift=-.05cm]n53.north);
\draw[semithick, ->] (n52.north) .. controls +(up:6mm) and +(up:6mm) .. ([xshift=1mm]n54.north);
\draw[semithick, ->] (n53.north) .. controls +(up:4mm) and +(up:4mm) .. ([xshift=-1mm]n54.north);

\node[draw, ultra thin] at (11, -1.75) {Step 4};
\node[roundnode][label={[label distance=-1cm]:$\phi(v_1)$}] (n61) at (\cccx, \rry)         {};
\node[roundnode][label={[label distance=-1cm]:$\phi(v_2)$}] (n62) at (\cccx+\xgap, \rry)   {};
\node[roundnode][label={[label distance=-1cm]:$\phi(v_3)$}] (n63) at (\cccx+2*\xgap, \rry) {};
\node[roundnode][label={[label distance=-1cm]:$\phi(v_4)$}] (n64) at (\cccx+3*\xgap, \rry) {};
\draw[semithick, ->] (n61.north) .. controls +(up:4mm) and +(up:4mm) .. ([xshift=-.05cm]n62.north);
\draw[semithick, ->] (n61.north) .. controls +(up:6mm) and +(up:6mm) .. ([xshift=-.05cm]n63.north);
\draw[semithick, ->] (n62.north) .. controls +(up:6mm) and +(up:6mm) .. ([xshift=.05cm]n64.north);
\draw[semithick, ->] (n63.north) .. controls +(up:4mm) and +(up:4mm) .. ([xshift=-.05cm]n64.north);

\end{tikzpicture}
    \caption{Illustration of the greedy embedding process for an acyclic orientation of the four-cycle. A directed arrow from a vertex to a set indicates that the vertex is complete to the set. All edges point forward in this example, but we do not always make this assumption.}
    \label{fig:greedy-embedding}
\end{figure}
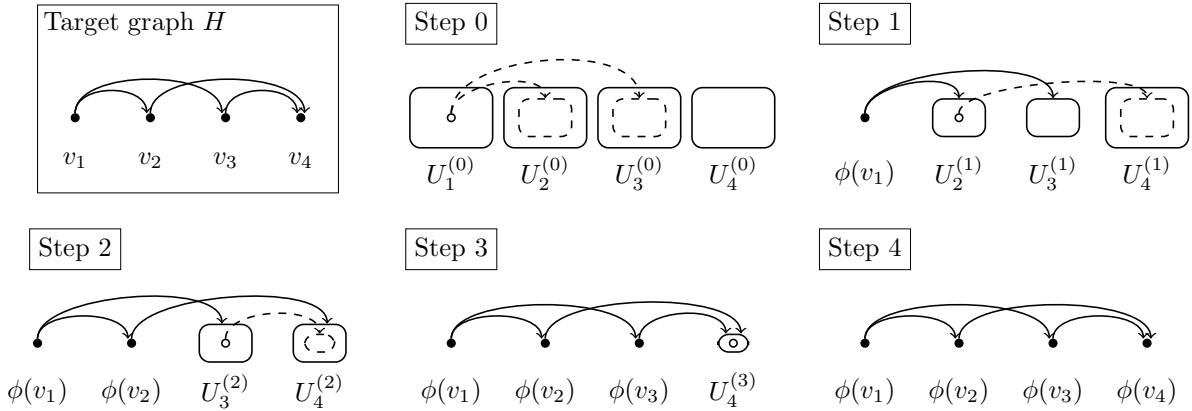

For the upper bounds described in the introduction (Theorems \ref{thm:G(n,d)-almost-linear}, \ref{thm:height-ub}, and \ref{thm:graded}), we apply greedy embedding in two separate stages, which
we call the outer stage and the inner stage. Roughly speaking, in
the outer stage, we run the greedy embedding procedure many times
to show that if $T$ is $H$-free, then $T$ contains a large ``approximate blowup of $H$.'' In the inner stage, we use greedy embedding
one final time within this ``approximate blowup'' to guarantee the
existence of a copy of $H$ in $T$. In either case, we can conclude that $T$ contains a copy of $H$---either it is found directly by the greedy embedding strategy, or else the failure of the greedy embedding yields the ``approximate blowup'' of $H$, in which a copy of $H$ can be found directly.

This section is split into three subsections. The first covers the basic results that follow from the greedy embedding framework described above, namely how a failure to greedily embed $H$ in a tournament $T$ implies that $T$ contains a certain nice structure, namely a pair of large vertex sets such that most edges between them have the same orientation. The tools built in this first subsection are then used as basic building blocks and iterated in the subsequent subsections. In the second subsection, we use them to build the outer stage of the embedding. In the third subsection, we explain how to use this outer stage construction of an ``approximate blowup of $H$'' to finally embed $H$ itself.

In some of these greedy embedding arguments, we are concerned
with partitioning an acyclic digraph $H$ into a number of parts and
embedding the parts one at a time, so the following definition will
be useful.
\begin{defn}
If $H$ is an acyclic digraph, we say that a collection $\{P_{i}\}_{i=1}^{r}$
of vertex subsets of $H$ is a \emph{directed partition} of $H$ if
$P_{1}\sqcup P_{2}\sqcup\cdots\sqcup P_{r}=V(H)$, and any edge of
$H$ between two distinct parts $P_{i},P_{j}$ with $i<j$ is oriented
from $P_{i}$ to $P_{j}$.
\end{defn}

In particular, the height of $H$ is exactly the least $h$ for
which there exists a directed partition of $H$ into $h$ independent sets. 

\subsection{The basic greedy embedding building blocks}
Recall that an undirected graph $G$ is said to be $d$\emph{-degenerate} if there
exists an ordering $v_{1},\ldots,v_{n}$ of the vertices of
$G$ such that each $|N(v_{i})\cap\{v_{1},\ldots,v_{i-1}\}|\le d$,
and such an ordering is called a $d$\emph{-degenerate ordering} of
$G$. A $d$-degenerate ordering is a natural order for greedily embedding
an undirected graph $G$, since each candidate set $U_{j}^{(t)}$ in the greedy embedding strategy
only shrinks in size by more than one at most $d$ times. We say that
a digraph $H$ is $d$-degenerate if its underlying undirected graph
is $d$-degenerate. Note that if $H$ has maximum degree
$\Delta$ then it is $\Delta$-degenerate, but $d$-degenerate digraphs
can have arbitrarily large maximum degree.

Define a $\delta$-\emph{dense pair} $(W_{1},W_{2})$ in a tournament
$T$ to be a pair of vertex subsets such that
at least $\delta|W_{1}||W_2|$ of the edges between $W_{1}$ and $W_{2}$
point from $W_{1}$ to $W_{2}$. The \emph{size} of the pair is defined
to be $\min(|W_{1}|,|W_2|)$. We do not require $W_1$ and $W_2$ to be disjoint, although the assumption of $\delta$-density implies that $W_1 \cap W_2$ cannot be too large if $\delta$ is close to $1$.

The first lemma in this subsection uses greedy embedding to show that
if $T$ doesn't contain a copy of a given $d$-degenerate $H$, then
$T$ contains some large dense pair. The undirected analogue of this lemma is well-known, and goes back at least to work of Erd\H os and Hajnal \cite[Lemma 1.5]{ErHa}.
\begin{lem}
\label{lem:dense-pair}Let $H$ be a $d$-degenerate digraph
with $n$ vertices and maximum degree $\Delta$, and let $0<c\le\frac{1}{2}$.
If $T$ is an $H$-free tournament on $N\geq2\Delta c^{-d}n$ vertices,
then $T$ contains a $(1-c)$-dense pair $(W_{1},W_{2})$ with size
at least $c^{d}N/(2\Delta)$.
\end{lem}

\begin{proof}
We use the greedy embedding framework described above. Let us
label the vertices of $H$ according to the $d$-degenerate ordering
as $v_{1},v_{2},\dotsc,v_{n}.$ We initialize $U_{i}^{(0)}=V(T)$
for all $1\leq i\leq n$. We now attempt to embed the vertices of
$H$ one at a time in $T$, in the $d$-degenerate ordering. For $j>t$,
let $N_{t}(v_{j})$ denote the set of vertices $v_{i}$ with $i\leq t$
such that $v_{i}$ and $v_{j}$ are connected by an edge (in some direction).
We inductively pick the values of $\phi(v_{t})\in V(T)$ and maintain
vertex sets $U_{i}^{(t)}$ with the following properties.
\begin{enumerate}
\item For every $i\leq t$, we have $U_{i}^{(t)}=\{\phi(v_{i})\}$. 
\item For every $j>t$, we have $|U_{j}^{(t)}|\geq c^{|N_{t}(v_{j})|}N-t$.
\item For $i\le t$, if $v_{i}\to v_{j}$ then $\phi(v_{i})\to x$ for every
$x\in U_{j}^{(t)}$, and if $v_{i}\leftarrow v_{j}$ then $\phi(v_{i})\leftarrow x$
for every $x\in U_{j}^{(t)}$.
\end{enumerate}
From these properties, we see that if the process continues through
step $t=n$, then we will have embedded a copy of $H$ in $T$, contradicting
our assumption that $T$ is $H$-free. Moreover, all three properties
are vacuously true for $t=0$. Now, suppose we have maintained this
process up through step $t-1$. Suppose there exists $w_{t}\in U_{t}^{(t-1)}$
such that for every $j>t$ with $v_t\to v_j$ (resp.\ $v_{t}\leftarrow v_j$),
we have $|N^{+}(w_{t})\cap U_{j}^{(t-1)}|\geq c|U_{j}^{(t-1)}|$ (resp.\ $|N^{-}(w_{t})\cap U_{j}^{(t-1)}|\geq c|U_{j}^{(t-1)}|$). Then we
may define $\phi(v_{t})=w_{t}$, $U_{t}^{(t)}=\{\phi(v_{t})\}$, and
update the remaining sets as 
\[
U_{j}^{(t)}\coloneqq\begin{cases}
[U_{j}^{(t-1)}\cap N^{+}(v_{t})]\setminus\{\phi(v_{t})\} & \text{if }v_{t}\rightarrow v_{j},\\{}
[U_{j}^{(t-1)}\cap N^{-}(v_{t})]\setminus\{\phi(v_{t})\} & \text{if }v_{j}\rightarrow v_{t},\\
U_{j}^{(t-1)}\setminus\{\phi(v_{t})\} & \text{else},
\end{cases}
\]
for all $j>t$. Properties 1 and 3 above continue to hold automatically
after round $t$, and all that remains to check is Property 2. For
those $j>t$ for which $v_{j}$ is not adjacent to $v_{t}$, $N_{t}(v_{j})=N_{t-1}(v_{j})$
and at most one vertex is removed from $U_{j}^{(t-1)}$ to obtain
$U_{j}^{(t)}$. Therefore, 
\[
|U_{j}^{(t)}|\geq|U_{j}^{(t-1)}|-1\geq c^{|N_{t-1}(v_{j})|}N-(t-1)-1=c^{|N_{t}(v_{j})|}-t.
\]
On the other hand, if $v_{j}$ is adjacent to $v_{t}$, then $|N_{t}(v_{j})|=|N_{t-1}(v_{j})|+1$,
and therefore
\[
|U_{j}^{(t)}|\geq c|U_{j}^{(t-1)}|-1\geq c\cdot(c^{|N_{t-1}(v_{j})|}N-(t-1))-1\geq c^{|N_{t}(v_{j})|}N-t.
\]
Thus, all three properties are maintained if such a $w_{t}$ exists.

Since we assumed that $T$ was $H$-free, this process cannot continue
until step $t=n$, and therefore must fail at some step $1\leq t\leq n-1$.
Let $W_{0}=U_{t}^{(t-1)}$. Since the process fails at this step,
for every $w\in W_{0}$, we can assign some $j>t$ such that either
$v_t \to v_j$ and $|N^{+}(w)\cap U_{j}^{(t-1)}|<c|U_{j}^{(t-1)}|$,
or $v_t \leftarrow v_j$ and $|N^{-}(w)\cap U_{j}^{(t-1)}|<c|U_{j}^{(t-1)}|$.
Since $v_{t}$ has at most $\Delta$ neighbors in total, at least
$|W_{0}|/\Delta$ choices of $w$ are assigned the same $j>t$ by
the pigeonhole principle. Fix such a ``popular'' $j$.

Suppose first that $v_t \to v_j$. Let $W_{2}=U_{j}^{(t-1)}$, and let $W_{1}$ be
the set of all $w\in W_0$ which have fewer than $c|W_{2}|$ out-neighbors
in $W_{2}$. Then $(W_{1},W_{2})$ is a $(1-c)$-dense pair. Similarly, if $v_{j}\in N^{-}(v_{t})$, then we would similarly find that $(W_2, W_1)$ is a $(1-c)$-dense pair.

It remains to verify the lower bound on the sizes of $W_{1}$ and
$W_{2}$. Recall that the greedy embedding process succeeded up through
step $t-1$, meaning that 
\[
|U_{t}^{(t-1)}|\geq c^{|N_{t-1}(v_{t})|}N-(t-1)\geq c^{d}N-n\geq\frac{c^{d}N}{2},
\]
and similarly for $U_{j}^{(t-1)}$, where we use the $d$-degeneracy
assumption to conclude that $|N_{t-1}(v_{t})|\leq d$, and our assumption
that $t<n\leq c^{d}N/2$. Since $|W_{2}|=|U_{j}^{(t-1)}|$ and $|W_{1}|\geq|U_{t}^{(t-1)}|/\Delta$,
this completes the proof.
\end{proof}
The second lemma proves a much stronger bound when $H$ is $1$-degenerate and weakly connected,
i.e.\ some orientation of a tree (recall that a digraph is called \emph{weakly connected} if its underlying undirected graph is connected). Note that a $1$-dense pair in a tournament is
just a pair of sets $W_{1},W_{2}$ with all edges
directed from $W_{1}$ to $W_{2}$. 
\begin{lem}
\label{lem:one-biclique-basic}Let $H$ be a weakly connected $1$-degenerate digraph with
maximum degree $\Delta$ on $m$ vertices $v_{1},\ldots,v_{m}$, and
let $T$ be an arbitrary tournament. If there exist sets $U_{1},\ldots,U_{m}\subseteq V(T)$, each
of size $M\geq 2m\Delta$ such that if there is no embedding $\phi:H\hookrightarrow T$
satisfying $\phi(v_{i})\in U_{i}$, then $T$ contains a $1$-dense
pair with size at least $M/(m(\Delta+1))$.
\end{lem}

\begin{proof}
We begin by picking subsets $V_1 \subseteq U_1,\dotsc,V_m \subseteq U_m$, each of size $M/m$, such that $V_1,\dotsc,V_m$ are pairwise disjoint. We can do this greedily, by first picking an arbitrary subset $V_1 \subseteq U_1$ of size $M/m$, then picking an arbitrary subset $V_2 \subseteq U_2 \setminus V_1$ of size $M/m$, then picking $V_3 \subseteq U_3 \setminus (V_1 \cup V_2)$, and so on. At the $i$th step, we have deleted at most $(i-1)M/m\leq (m-1)M/m$ elements from $U_i$, so at least $M/m$ elements remain, from which we pick $V_i$ arbitrarily.

The remainder of the proof is very similar to that of the previous lemma, though it will be
more convenient to work with a slightly different setup than before.
Let the vertices of $H$ be ordered so that each $v_{i}$ has at most
one neighbor $v_{j}$ with $j>i$ (this is the reverse of the usual
degenerate ordering). For each $t$, let $S_{t}$ denote the set of
vertices $v_{i}$ with $i\le t$ that are connected to $v_i$ by a path
of vertices whose indices are monotonically increasing, including
$v_{t}$ itself. In other words, we set $S_{1}=\{v_{1}\}$, let $N_{t}(v_{t})$
denote the set of vertices $v_{i}$ adjacent to $v_{t}$ (in either
orientation) with $i<t$, and set
\[
S_{t}=\{v_{t}\}\cup\bigcup_{i\in N_{t}(v_{t})}S_{i}.
\]

We define $W_{t}\subseteq V_{t}$ to be the set of all $w\in V_{t}$
for which there exists an embedding $\phi:H[S_{t}]\hookrightarrow T$
mapping each $v_{i}$ into $V_{i}$ for $v_{i}\in S_{t}$, and mapping $v_t$ to $w$. We know
$W_{n}=\varnothing$, since otherwise there is an embedding of $H$
into the sets $V_{1},\dotsc,V_{m}$, contradicting our assumption.
Let $t\le n$ be the minimum index such that $|W_{t}|<M/(m(\Delta+1))$;
note that $t>1$ since $W_{1}=V_{1}$ has size $M/m$. Let $N_{t}(v_{t})=\{v_{i_{1}},\dotsc,v_{i_{s}}\}$.
Our maximum degree assumption implies $s\le\Delta$, while if $s=0$,
then $W_{t}=V_{t}$, contradicting our assumption that $|W_{t}|<M/(m(\Delta+1))$.
Thus, $1\le s\le\Delta$. 

By definition, $W_{t}$ is precisely the set of $w\in V_{t}$ adjacent
in the appropriate orientation to at least one vertex in each of $W_{i_{1}},\dotsc,W_{i_{s}}$.
Let $X_{1},\dotsc,X_{s}\subseteq V_{t}$ denote the choices of $w$
which do not have any edge in the appropriate orientation to $W_{i_{1}},\dotsc,W_{i_{s}}$,
respectively. We get $V_{t}=W_{t}\cup X_{1}\cup\dotsb\cup X_{s}$.
Since $s\leq\Delta$ and $|W_{t}|\leq M/(m(\Delta+1))$, we see that
there exists some $j$ for which $|X_{j}|\geq M/(m(\Delta+1))$. Moreover,
since $t$ was taken to be minimal, we have that $|W_{i_{j}}|\geq M/(m(\Delta+1))$
as well. This yields a pair of sets $(X_{j},W_{i_{j}})$ where all
edges between them are oriented the same way, and both sets have size
at least $M/(m(\Delta+1))$. This is the desired $1$-dense pair.
\end{proof}
We remark that we expect the $m$ dependence in Lemma \ref{lem:one-biclique-basic} to be unnecessary, and proving this would improve our results for random digraphs; for details, see Conjecture \ref{conj:better-biclique} and the surrounding discussion. 

We do not use Lemma \ref{lem:one-biclique-basic} directly, but only the following simple corollary. It allows us to find a large $1$-dense pair in $T$ whenever $T$ does not contain a copy of some fixed oriented forest with small maximum degree and small weakly connected components.
\begin{lem}\label{lem:one-biclique}
Let $H$ be a $1$-degenerate digraph with maximum degree $\Delta$ and vertices $v_1,\dotsc,v_n$, and suppose that every weakly connected component of $H$ has at most $m$ vertices. Let $T$ be an arbitrary tournament. For any collection of sets $V_{1},\ldots,V_{n}\subseteq V(T)$, each
of size $M\geq 3n\Delta$, either there is an embedding $\phi:H\hookrightarrow T$
satisfying $\phi(v_{i})\in V_{i}$, or $T$ contains a $1$-dense
pair with size at least $M/(4m\Delta)$.
\end{lem}
\begin{proof}
Let the weakly connected components of $H$ be $C_1,\dotsc,C_r$. We prove the result by induction on $r$. The base case, $r=1$, follows immediately from Lemma \ref{lem:one-biclique-basic}, since $n \geq m$ and $\Delta+1 \leq 2\Delta$. For the inductive step, let $H'$ be the induced subgraph of $H$ consisting of the weakly connected components $C_1,\dotsc,C_{r-1}$. By the inductive hypothesis, either $T$ contains a $1$-dense pair of size at least $M/(4m \Delta)$, in which case we are done, or else there is an embedding $\phi$ of $H'$ into $T$ satisfying that $\phi(v_i) \in V_i$ for all $v_i \in V(H')$. Let $T'$ be the tournament obtained by deleting the image of $\phi$ from $T$, and similarly let $V_1',\dotsc,V_n'$ be obtained by deleting the image of $\phi$ from $V_1,\dotsc,V_n$. Since we have deleted at most $n \leq n \Delta$ vertices, each $V_i'$ has size at least $M' \geq 2n \Delta\geq 2m \Delta$. Therefore, by Lemma \ref{lem:one-biclique-basic}, either $T'$ contains a $1$-dense pair of size at least $M'/(2m \Delta)\geq M/(4m \Delta)$, in which case so does $T$, or else there is an embedding of $H[C_r]$ into $T'$ mapping each $v_i \in C_r$ into $V_i'$. Since we deleted the image of $\phi$ from $T$ to define $T'$, such an embedding yields an embedding of $H$ into $T$, completing the induction.
\end{proof}

Our final basic greedy embedding lemma is the following, which shows that if we assume appropriate density conditions on the sets in which we are trying to greedily embed $H$, then we are guaranteed not to fail in the embedding. In Subsection \ref{subsec:inner-stage}, we use this lemma as a basic building block. Again, the undirected analogue of this lemma is well-known, e.g.\ \cite[Lemma 2]{GrRoRu}.
\begin{lem}
\label{lem:inner-stage}Let $H$ be an acyclic digraph with maximum
degree $\Delta$ on $n$ vertices $v_{1},\ldots,v_{n}$, and let $T$
be a tournament containing subsets $V_{1},\ldots,V_{n}\subseteq V(T)$,
each of size at least $4n$. If for every edge $v_{i}\rightarrow v_{j}$
in $H$, $(V_{i},V_{j})$ is a $(1-\frac{1}{8\Delta^{2}})$-dense
pair, then $T$ contains a copy of $H$.
\end{lem}

Note that the sets $V_{i}$ are allowed to overlap\footnote{However, if $v_i \to v_j$, then the assumption that $(V_i,V_j)$ is $(1- \frac{1}{8 \Delta^2})$-dense implies that $V_i\cap V_j$ cannot be too large.} or even be identical.
This is crucial; for instance, if $H$ has height $h$ then our choice
of $V_{i}$'s will take only $h$ distinct values.
\begin{proof}
Every acyclic digraph has a vertex ordering where all edges point
forward, so we may reorder the vertices to assume all edges $v_{i}\rightarrow v_{j}$
satisfy $i<j$. We now run the greedy embedding process for $H$
into $T$ with the given $U_{i}$, using essentially the same framework
as we used in Lemma \ref{lem:dense-pair}. The key difference is that before,
the greedy embedding process could fail and terminate prematurely,
whereas the additional assumptions here guarantee that greedy embedding
runs to completion.

We begin by refining the sets $V_{1},\dotsc,V_{n}$. Namely, for every
edge $v_{i}\to v_{j}$ in $H$, let $V_{i,j}\subseteq V_{i}$ be the
set of vertices $w\in V_{i}$ with $|N^{-}(w)\cap V_{j}|\ge|V_{j}|/(4\Delta)$.
Since the pair $(V_{i,},V_{j})$ is $(1-\frac{1}{8\Delta^{2}})$-dense,
there are at most $|V_{i}||V_{j}|/(8\Delta^{2})$ edges directed from
$V_{j}$ to $V_{i}$ in total, so
\[
|V_{i,j}|\cdot\frac{|V_{j}|}{4\Delta}\leq\frac{|V_{i}||V_{j}|}{8\Delta^{2}},
\]
and thus $|V_{i,j}|\leq|V_{i}|/(2\Delta)$. Each $v_{i}$ has at most
$\Delta$ out-neighbors $v_{j}$, which implies that $|\bigcup_{j}V_{i,j}|\leq|V_{i}|/2$.
Therefore, if $U_{i}\coloneqq V_{i}\setminus(\bigcup_{j}V_{i,j})$,
then $|U_{i}|\geq|V_{i}|/2$, and every vertex in $U_{i}$ has
at most $|V_{j}|/(4\Delta)\leq|U_{j}|/(2\Delta)$ in-neighbors in
$U_{j}$ for every $j$ such that $v_{i}\to v_{j}$.

Now, we exhibit an embedding $\phi:H\hookrightarrow T$
by picking $\phi(v_{t})\in U_{t}$ inductively for each $t\in[n]$,
as follows. Having picked $\phi(v_{1}),\ldots,\phi(v_{t-1})$, we
claim that there is at least one valid candidate for $\phi(v_{t})\in U_{t}$
which is consistent with the previous choices. Indeed, there are at
most $\Delta$ choices of $i<t$ such that $v_{i}\to v_{t}$. For
every such $i$, we have that $\phi(v_{i})\in U_{i}$, which means
that $\phi(v_{i})$ has at most $|U_{t}|/(2\Delta)$ in-neighbors
in $U_{t}$. Hence, at least $|U_{t}|/2$ vertices in $U_{t}$ are
out-neighbors of $\phi(v_{i})$ for all $i$ such that $v_{i}\to v_{j}$.
Moreover, at most $t-1$ vertices of $U_{j}$ have been picked as
outputs of $\phi$, and thus the number of possible candidates for
$\phi(v_{t})$ is at least
\[
\frac{|U_{t}|}{2}-(t-1)\geq\frac{|V_{t}|}{4}-n+1\geq1,
\]
by our assumption that $|V_{j}|\geq4n$. This shows that at every
step $t\le n$ we can pick vertex $\phi(v_{t})$ such that $\{\phi(v_{1}),\ldots,\phi(v_{t})\}$
is a copy of $H[\{v_{1},\ldots,v_{t}\}]$ in $T$, and so $T$ contains
a copy of the entirety of $H$, as desired.
\end{proof}
We remark that one can prove a strengthening of Lemma 2.4, where the assumption is weakened to each pair $(V_i, V_j)$ being merely $(1-\Omega(\frac 1\Delta))$-dense. This can be done by replacing the greedy embedding argument by a random embedding technique, using the Lov\'asz local lemma. For details, we refer the reader to \cite[Lemma 4.5]{CoFoSuII}, where the analogous undirected result is proved using this technique.

\subsection{The outer stage}

In this subsection, we show how two of our basic building blocks, Lemmas \ref{lem:dense-pair} and \ref{lem:one-biclique}, can be iterated to construct, in any $H$-free tournament $T$, an ``approximate blowup'' of $H$. This will be a large collection of vertex sets which correspond to the vertices of $H$, such that the edges between sets are mostly oriented in the correct direction. Before stating the result precisely, we need some definitions.

Let $\{0,1\}^{*}$
denote the set of all finite binary strings. Recall that a \emph{prefix
code} is a set $C\subset\{0,1\}^{*}$ with the property that no element
of $C$ is a prefix of another element of $C$. The \emph{depth }of $C$ is
defined as the maximum length of an element of $C$. Let $\prec$ denote the lexicographic ordering on $\{0,1\}^*$, namely the ordering in which $x\prec y$ if $x$ is a proper prefix of $y$ or if $x_i < y_i$ where $i$ is the first index for which $x_i \neq y_i$.
\begin{defn}
Given an acyclic digraph $H$, a \emph{prefix labeling }of $H$ is
a surjective map $\rho:V(H)\to C$ for some prefix code $C\subset\{0,1\}^{*}$,
with the property that if $v_{i}\to v_{j}$ is an edge of $H$, then
either $\rho(v_{i})=\rho(v_{j})$ or $\rho(v_{i})\prec\rho(v_{j})$. The map $\rho$ naturally defines a graph structure on $C$, where we say that two codewords $x,y$ are \emph{adjacent under $\rho$} if there exists some edge between the sets $\rho^{-1}(x)$ and $\rho^{-1}(y)$. By the \emph{maximum degree of $\rho$}, we mean the maximum degree of this graph on $C$. If $\rho^{-1}(x)$ is an independent set for every $x\in C$, then
we call $\rho$ a \emph{prefix coloring. }Less stringently, we call
$\rho$ a \emph{forest prefix labeling} if $\rho^{-1}(x)$ is a directed
forest (or equivalently, a $1$-degenerate digraph) for every $x\in C$. By the \emph{maximum component size} of $\rho$, we mean the maximum number of vertices of any weakly connected component in $\rho^{-1}(x)$, over all $x \in C$. Thus, $\rho$ is a prefix coloring if and only if its maximum component size is $1$.
\end{defn}

Thus, we see that prefix colorings of $H$ correspond to colorings
of the underlying undirected graph of $H$, with the property that
the palette of colors $C$ is a prefix code, and that the lexicographic
order on $C$ is consistent with the edge directions in $H$. Similarly,
a forest prefix labeling is in particular a partition of the underlying
undirected graph into sets which induce forests, which corresponds
to the undirected problem of \emph{vertex arboricity. }However, for
both concepts, we will crucially use the additional structure given
both by the edge directions of $H$ and by the structure of the prefix
code $C$. 

For a binary string $x$, let us denote by $x0$ and $x1$ the
strings obtained by appending a $0$ or $1$, respectively, to the end of $x$.
For a prefix code $C$, we denote by $C_{0}(x)$ the set of all elements
$y\in C$ which have $x0$ as a prefix, and by $C_{1}(x)$
the set of elements that have $x1$ as a prefix. Suppose that
$\rho:V(H)\to C$ is a prefix labeling of an acyclic digraph $H$.
For binary string $x$, let us denote by $a_{0}(x)$ the number of
codewords $y\in C_{0}(x)$ which are adjacent under $\rho$ to
some codeword $z\in C_{1}(x)$. Similarly, $a_{1}(x)$ is the number
of codewords $z\in C_{1}(x)$ that are adjacent under $\rho$ to some
codeword $y\in C_{0}(x)$. Finally, we let
\[
a(x)=\begin{cases}
1 & \text{if }a_{0}(x)=0\text{ or }a_{1}(x)=0,\\
\min(a_{0}(x),a_{1}(x)) & \text{otherwise.}
\end{cases}
\]
In particular, $a(x)=1$ if $x$ is not a proper prefix of any element
of $C$. With this notation, we can now define the key parameters
of $\rho$ that we later use to bound Ramsey numbers. 
\begin{defn}
Let $H$ be an acyclic digraph and $\rho:V(H)\to C$ some prefix labeling
of $H$. We define the \emph{dyadic complexity }of $\rho$ to be the
quantity
\[
\comp(\rho):=\max_{y\in C}\prod_{x\text{ a prefix of }y}a(x).
\]
Additionally, the \emph{depth} of $\rho$, denoted $\depth(\rho)$,
is defined as the depth of $C$, i.e.\ the maximum length of
an element of $C$. 
\end{defn}
To understand these definitions, it is helpful to think of $\{0,1\}^*$ as the vertices of the infinite binary tree. In this setup, the elements of a prefix code $C$ correspond to the leaves of some subtree. A prefix labeling $\rho:V(H) \to C$ is then a partition of the vertices of $H$ into sets labeled by the leaves of this subtree. Two codewords (leaves) are adjacent under $\rho$ if there is an edge between the corresponding vertex subsets of $H$. For a binary string $x$, which should be thought of as a non-leaf vertex of the subtree, $a(x)$ roughly records the ``cost'' of separating the descendants of $x$: it measures how many pairs of codewords adjacent under $\rho$ there are between its descendants on the left and on the right. Because of the structure of the proof, this cost function is somewhat unnatural: it is the minimum of two quantities, each of which is the number of descendants on one side which are adjacent under $\rho$ to any number of descendants on the other side. Moreover, this cost function should be thought of as multiplicative, so that the ultimate cost of the whole labeling---namely the dyadic complexity $\comp(\rho)$---is the product of the costs of all ancestors of $y$, maximized over all $y \in C$.

We remark that the dyadic complexity of a prefix coloring of $H$ is one possible formalization of the notion of multiscale complexity discussed in the introduction. Indeed, if every prefix coloring of $H$ has high dyadic complexity, then $H$ has many edges at ``many different dyadic scales''. All of our upper bounds on oriented Ramsey numbers depend on the dyadic complexity of some prefix labeling of $H$, making formal the intuition that the strength of our upper bound results depends on whether $H$ has high or low multiscale complexity.

In order to embed some acyclic digraph $H$ in a tournament $T$,
we first build a certain structure of vertex subsets of $T$
with high forward edge density between many of the pairs. Then, in
the inner stage of the embedding process, we use such
a structure to find a copy of $H$. The structure we build depends
on a parameter $\delta\in[0,1]$, as well as on a prefix labeling
$\rho$ of $H$. We now define this structure, and next prove
a lemma showing how to find such a structure.
\begin{defn}
Let $\delta\in[0,1]$ be some parameter, let $H$ be an acyclic digraph,
and let $\rho:V(H)\to C$ be some prefix labeling of $H$, for some
prefix code $C$. Let $T$ be any tournament. A \emph{$(\rho,\delta)$-skeleton
}is a collection $\{V_{x}\}_{x\in C}$ of (not necessarily disjoint)
vertex subsets of $T$, indexed by the codewords in $C$, with the
property that if $x\prec y$ are elements of $C$ that are adjacent under $\rho$, then $(V_{x},V_{y})$
is a $\delta$-dense pair. We define the \emph{size }of a $(\rho,\delta)$-skeleton
to be $\min_{x\in C}|V_{x}|$. 
\end{defn}

Our next lemma shows how to iterate Lemma \ref{lem:dense-pair} to
construct a $(\rho,\delta)$-skeleton in any sufficiently large $H$-free
tournament $T$. Roughly speaking, since the structure of a $(\rho,\delta)$-skeleton
is based on the binary tree structure of $C$, we may construct such
a skeleton by performing a depth-first search, and applying Lemma
\ref{lem:dense-pair} every time we need to split an existing node
into two daughter nodes in this binary tree. 
\begin{lem}
\label{lem:build-skeleton} Let $c\in(0,1)$ be some parameter, let
$H$ be a $d$-degenerate acyclic digraph with maximum degree $\Delta$,
and let $\rho:V(H)\to C$ be some prefix labeling of $H$, for some
prefix code $C$. Suppose that $T$ is an $H$-free tournament on $N$ vertices, with
\[
N\geq\left(\frac{4^{d+1}\Delta}{c^{d}}\right)^{\depth(\rho)}\comp(\rho)^{d}n.
\]
Then $T$ contains a $(\rho,1-c)$-skeleton of size at
least $\left(\frac{c^{d}}{4^{d+1}\Delta}\right)^{\depth(\rho)}\comp(\rho)^{-d}N$.
\end{lem}

\begin{proof}
For every binary string $x\in\{0,1\}^{*}$ which is a prefix of some
codeword in $C$, let $H_{x}$ denote the subgraph of $H$ induced
by the vertices $v$ for which $x$ is a prefix of $\rho(v)$. We
will construct, for every such $x\in\{0,1\}^{*}$, a vertex subset
$V_{x}\subseteq V(T)$, with
\[
|V_{x}|\geq\left(\prod_{z\text{ a proper prefix of }x}\frac{c^{d}}{4^{d+1}\Delta a(z)^{d}}\right)N.
\]
We initialize $V_{\varnothing}=V(T)$, where $\varnothing\in\{0,1\}^{*}$
denotes the empty string, and observe that this property holds vacuously
for $V_{\varnothing}$ since $\varnothing$ has no proper prefixes.
In order to construct these vertex sets for other $x$, we proceed
via a depth-first search along the binary tree, as follows. Recall
that for any binary string $x$, the numbers $a_{0}(x)$ and $a_{1}(x)$
are the number of codewords in $C$ beginning with $x0$ and $x1$,
respectively, such that some vertex labeled by that codeword has an
edge to $H_{x1}$ and $H_{x0}$, respectively. If $\min(a_{0}(x),a_{1}(x))=0$,
then we define $V_{x0}=V_{x1}=V_{x}$, and observe that our
desired inequality holds, since $a(x)=1$ in this case. In particular, if $x\in C$ is a codeword,
then we stop the inductive process, since we only wished to define
such vertex subsets for strings $x$ that are the prefix of some codeword.
Now, suppose that $\min(a_{0}(x),a_{1}(x))\geq1$, and assume without
loss of generality\footnote{If $a_1(x) < a_0(x)$, then we swap the roles of $0$ and $1$ and the roles of in- and out-neighbors in this construction.} that $a_{0}(x)\leq a_{1}(x)$. We will first show how to define $V_{x0} \subseteq V_x$ satisfying the desired inequality on its cardinality. We will then proceed to recursively define vertex subsets $V_{y}\subseteq V_{x0}$
for every binary string $y$ prefixed by $x0$. Note that we have
not yet defined the set $V_{x1}$: we are proceeding in a depth-first
fashion, so we will not define $V_{x1}$ until we have defined
$V_{y}$ for every $y$ prefixed by $x0$. This will eventually
happen, since we already described above how to define $V_{y}$ if
$y$ is a codeword of $C$; therefore, we eventually reach the
bottom of our depth-first search (namely a codeword $y\in C$), at
which point we stop going down the tree, and begin to retrace
our steps and traverse back up the tree.

Recall that we assumed that $a_0(x) \leq a_1(x)$, and let $c_{x}=c/(4a_{0}(x)).$ Since 
\begin{align}
|V_{x}|&\geq\left(\prod_{z\text{ a proper prefix of }x}\frac{c^{d}}{4^{d+1}\Delta a(z)^{d}}\right)N\notag\\
&=\frac{4\Delta (4a(x))^d}{c^d} \left(\prod_{z\text{ a proper prefix of }x0}\frac{c^{d}}{4^{d+1}\Delta a(z)^{d}}\right)N\notag
\\
&\geq\frac{\left(c^{d}/(4^{d+1}\Delta)\right)^{\depth(\rho)}\comp(\rho)^{-d}}{c^{d}/(4\Delta(4a(x))^{d})}N\notag\\
&\geq4\Delta c_{x}^{-d}n,\label{eq:V_x-lb}
\end{align}
and since $T$ contains no copy of $H$, we may apply Lemma \ref{lem:dense-pair}
with parameter $c_{x}$. Then this lemma says that $V_{x}$ contains
a $(1-c_{x})$-dense pair $(W_{0},W_{1})$, where $\min(|W_{0}|,|W_{1}|)\geq c_{x}^{d}|V_{x}|/(2\Delta)$.
Let $V_{x0}\subseteq W_{0}$ denote the set of vertices in $W_{0}$
whose in-degree to $W_{1}$ is at most $2c_{x}|W_{1}|$; since there
are at most $(1-c_{x})|W_{0}||W_{1}|$ edges directed from $W_{1}$
to $W_{0}$, we see that $|V_{x0}|\geq|W_{0}|/2$. Therefore,
\begin{equation}
|V_{x0}|\geq\frac{c_{x}^{d}}{4\Delta}|V_{x}|=\frac{c^{d}}{4^{d+1}\Delta a(x)^{d}}|V_{x}|\geq\left(\prod_{z\text{ a proper prefix of }x0}\frac{c^{d}}{4^{d+1}\Delta a(z)^{d}}\right)N,\label{eq:inductive-dyadic-lb}
\end{equation}
since the proper prefixes of $x0$ are just $x$, in addition to
all the proper prefixes of $x$ itself. 

As discussed above, we can now recursively define $V_y$ for all $y$ prefixed by $x0$. It thus only
remains to define $V_{x1}\subseteq V_{x}$, under the assumption
that we have defined $V_{y}\subseteq V_{x0}\subseteq W_{0}$ for
all $y$ prefixed by $x0$. 

To do so, let $y_{1},\dotsc,y_{a_{0}(x)}$ denote the set of codewords
prefixed by $x0$ with the property that some vertex in $\rho^{-1}(y_{i})$
is adjacent to some vertex in $H_{x1}$, noting that there are
exactly $a_{0}(x)$ such codewords by the definition of $a_{0}(x)$.
Let $W_{1}^{(i)}$ denote the subset of $W_{1}$ consisting of vertices
in $W_{1}$ whose out-degree to $V_{y_{i}}$ is at least $c|V_{y_{i}}|$.
Since every vertex in $V_{y_{i}}\subseteq V_{x0}$ has at most
$2c_{x}|W_{1}|$ in-neighbors in $W_{1}$, we see that the total number
of edges directed from $W_{1}$ to $V_{y_{i}}$ is at most $2c_{x}|W_{1}||V_{y_{i}}|$,
and therefore $|W_{1}^{(i)}|\leq(2c_{x}/c)|W_{1}|$. Because of this,
we have that
\[
\left|\bigcup_{i=1}^{a_{0}(x)}W_{1}^{(i)}\right|\leq\sum_{i=1}^{a_{0}(x)}|W_{1}^{(i)}|\leq a_{0}(x)\frac{2c_{x}}{c}|W_{1}|=\frac{|W_{1}|}{2}.
\]
Therefore, we can define $V_{x1}=W_{1}\setminus\left(\bigcup_{i=1}^{a_{0}(x)}W_{1}^{(i)}\right)$,
and we have that $|V_{x1}|\geq|W_{1}|/2$. By the same computation
as in equation (\ref{eq:inductive-dyadic-lb}), we see that this definition
of $V_{x1}$ satisfies our desired lower bound on the size of $V_{x1}$. 

We claim that in this construction, if $y\prec z$ are codewords that are adjacent under $\rho$,
then the pair $(V_{y},V_{z})$ is $(1-c)$-dense. To see this, let
$x$ be the longest common prefix of $y$ and $z$. In the construction
at level $x$, we first proceeded to either $V_{x0}$ or $V_{x1}$
in the depth-first search, depending on the relative sizes of $a_{0}(x)$
and $a_{1}(x)$. In the former case, we ensured that every vertex
in $V_{x1}$ had out-degree at most $c|V_{y}|$ to $V_{y}$, while
in the latter case, we ensured that every vertex in $V_{x0}$ had
in-degree at most $c|V_{z}|$ to $V_{z}$. In either case, we see
that $(V_{y},V_{z})$ is $(1-c)$-dense, since $V_y \subseteq V_{x0}$ and $V_z \subseteq V_{x1}$. Additionally, by the same
computation as in equation (\ref{eq:V_x-lb}), we see that $|V_{y}|\geq\left(\frac{c^{d}}{4^{d+1}\Delta}\right)^{\depth(\rho)}\comp(\rho)^{-d}N$
for every $y\in C$. This verifies all the properties of a $(\rho,1-c)$-skeleton,
and concludes the proof.
\end{proof}

Our next two lemmas are very similar to Lemmas \ref{lem:dense-pair} and \ref{lem:build-skeleton}. Namely, the first shows us how to find a $1$-dense pair in an $H$-free tournament $T$, and the second then iterates the first to form a skeleton of $1$-dense pairs. The main difference between these and the previous results are that for these lemmas, we need strengthened assumptions on $H$ (namely that it has a directed partition into forests). Moreover, the first step in the proof is an application of Lemma \ref{lem:build-skeleton}, and we find the $1$-dense pair by failing to greedily embed $H$ in the skeleton given by Lemma \ref{lem:build-skeleton}.

\begin{lem}
\label{lem:one-biclique-again} Let $d \geq 2$, and suppose that $H$ is a $d$-degenerate
acyclic digraph on $n$ vertices with maximum degree $\Delta$, and
suppose that $\rho:V(H)\to C$ is some forest prefix labeling with
maximum degree $A$ and maximum component size $m$.
Let $T$ be an $H$-free tournament on $N\geq(2^{10}A\Delta^{2})^{d\depth(\rho)}\comp(\rho)^{d}n$
vertices. Then $T$ contains a $1$-dense pair $(U_{1},U_{2})$ of
size at least $m^{-1}(2^{10}A\Delta^{2})^{-d\depth(\rho)}\comp(\rho)^{-d}N$.
\end{lem}

\begin{proof}
We first
apply Lemma \ref{lem:build-skeleton} to the prefix labeling $\rho$
and with parameter $c=1/(16A\Delta)$. This lemma outputs
a $(\rho,1-\frac{1}{16A\Delta})$-skeleton $\{V_x\}_{x \in C}$ in $T$ of size at least
\[
\left(\frac{c^{d}}{4^{d+1}\Delta}\right)^{\depth(\rho)}\comp(\rho)^{-d}N\geq 12 \Delta n.
\]
For every pair of codewords $x\prec y$ that are adjacent under $\rho$, let $V_{x,y}$
denote the set of vertices in $V_{x}$ whose in-degree to $V_{y}$
is at least $\frac{1}{8\Delta}|V_{y}|$. Since at most $\frac{1}{16A\Delta}|V_{x}||V_{y}|$
of the edges between $V_{x}$ and $V_{y}$ are directed from $V_{y}$
to $V_{x}$, we see that $|V_{x,y}|\leq\frac{1}{2A}|V_{x}|$. Therefore,
if we define $V_{x}'=V_{x}\setminus\left(\bigcup_{y\succ x}V_{x,y}\right)$,
then we see that $|V_{x}'|\geq \frac 12|V_{x}|$, since there are at most $A$ choices for $y \succ x$ with $x,y$ adjacent under $\rho$. Additionally, every vertex in
$V_{x}'$ has in-degree at most $\frac{1}{8\Delta}|V_{y}|\leq\frac{1}{4\Delta}|V_{y}'|$
from any $V_{y}'$ with $y\succ x$ such that $x,y$ are adjacent under $\rho$.
We now attempt to greedily embed $H$ in these sets $\{V_{x}'\}_{x\in C}$. 

Let the codewords of $C$ be $x_{1},\dotsc,x_{r}$, sorted so that
$i<j$ if and only if $x_{i}\prec x_{j}$. Let $P_{i}=\rho^{-1}(x_{i})$,
so that $P_i$ is an oriented forest, each of whose weakly connected components has at most $m$ vertices. Additionally, $P_1\sqcup \dotsb \sqcup P_r$ forms a directed partition of $V(H)$, which we recall means that every edge of $H$ is oriented from $P_i$ to $P_j$ where $i \leq j$. For every vertex $v\in P_{i}$, we initialize
a set of candidates $U_{i,v}^{(0)}=V_{x_{i}}'$. Inductively, having
defined $U_{i,v}^{(t)}$ for each $i>t$ and each $v\in P_{i}$, we
attempt to pick an embedding $\phi_{t}:H[P_{t}]\hookrightarrow V_{x_{t}}'$,
such that for every $v\in P_{t}$, $\phi_{t}(v)\in U_{t,v}^{(t-1)}$.
If such a $\phi_{t}$ exists, then for each $i>t$ and each $v\in P_{i}$,
we let 
\[
U_{i,v}^{(t)}\coloneqq\{u\in U_{i,v}^{(t-1)}\backslash\phi_{t}(P_{t}):{\forall w\in N^{-}(v) \cap P_t,}\phi_{t}(w)\rightarrow u\}.
\]
Note that by the structure of the sets $V_{x_{1}}',\dotsc,V_{x_{r}}'$,
we only change $U_{i,v}^{(t)}$ as follows. First, in at most $\Delta$ steps, we embed an in-neighbor of $v$, and we decrease $|U_{i,v}^{(t)}|$ by at most $\frac{1}{4\Delta}|V_{x_{i}}'|$. Additionally, we remove at most $n$ additional vertices from $U_{i,v}^{(t)}$, corresponding to vertices that were picked as images of $\phi_t$. In total, we remove at most $\Delta \cdot \frac{1}{4\Delta}|V_{x_{i}}'| +n\leq \frac{1}{2}|V_{x_i}'|$ vertices.
We thus see that $|U_{i,v}^{(t)}|\geq \frac 12|V_{x_{i}}'|$ for all
$t$. 

If we are able to run this process for all $1\leq t\leq r$, then
we have embedded a copy of $H$ in $T$, so we may assume that the
process fails at some step $t$. This means that there is no embedding
$\phi_{t}:H[P_{t}]\hookrightarrow T$ such that every vertex
$v\in P_{t}$ is mapped to $U_{t,v}^{(t-1)}$. Therefore, by applying Lemma
\ref{lem:one-biclique} to $H[P_t]$, we conclude that $T$ contains a $1$-dense pair $(U_{1},U_{2})$
of size at least 
\[
\frac{|U_{t,v}^{(t-1)}|}{4m\Delta}\geq\frac{|V_{x_{t}}|}{16m\Delta}\geq \frac 1m (2^{10}A\Delta^{2})^{-d\depth(\rho)}\comp(\rho)^{-d}N.\qedhere
\]
\end{proof}
Our next result shows how we can iterate the previous lemma to construct
many $1$-dense pairs. The iteration is nearly identical to the one
in Lemma \ref{lem:build-skeleton}, where we iterated the construction
of one dense pair to the construction of a $(\rho,1-c)$-skeleton.
This lemma takes as input two (not necessarily distinct) prefix labelings
on $H$, one of which is a forest labeling. The reason to have two
separate labelings is that it may be useful to use the failure of
embedding of $H$ according to one labeling to construct a good embedding
structure for the other labeling.
\begin{lem}
\label{lem:biclique-skeleton}Let $d \geq 2$, and suppose that $H$ is a $d$-degenerate
acyclic digraph on $n$ vertices with maximum degree $\Delta$. Let
$\rho_{1}:V(H)\to C_{1}$ and $\rho_{2}:V(H)\to C_{2}$ be two prefix
labelings, such that $\rho_{1}$ is a forest prefix labeling of maximum
degree $A_{1}$ and maximum component size $m_1$. Let
\[
\gamma = \left(m_1(2^{10}A_{1}\Delta^{2})^{d\depth(\rho_{1})}\comp(\rho_{1})^d\right) ^{-1}.
\]
If $T$ is an $H$-free tournament on $N\geq\gamma^{-\depth(\rho_{2})}n$
vertices, then $T$ contains a $(\rho_{2},1)$-skeleton of size at
least $\gamma^{\depth(\rho_2)}N$.
\end{lem}

\begin{proof}
As in the proof of Lemma \ref{lem:build-skeleton}, we construct
our skeleton by assigning a set $V_{x}\subseteq V(T)$ for
every $x\in\{0,1\}^{*}$ that is a prefix of some codeword in $C_{2}$, with the property that $V_{x0}, V_{x1}$ are both subsets of $V_x$.
We guarantee inductively that
\begin{equation}
|V_{x}|\geq \gamma^{|x|}N,\label{eq:biclique-skeleton-size-lb}
\end{equation}
where $|x|$ is the length of $x$. To begin the induction, we set
$V_{\varnothing}=V(T)$, which satisfies our size hypothesis since
$|\varnothing|=0$. Inductively, suppose we've defined $V_{x}$. If
$x\in C_{2}$, we stop. If not, we apply Lemma \ref{lem:one-biclique-again}
to the induced subtournament on $V_{x}$, which we may do since
\begin{align*}
|V_{x}|&\geq\gamma^{|x|}N\geq\gamma^{\depth(\rho_{2})-1}N\geq(2^{10}A_{1}\Delta^{2})^{d\depth(\rho_1)}\comp(\rho_1)^{d}n.
\end{align*}
This allows us to find a $1$-dense pair $(U_{1},U_{2})$ of size
at least $\gamma|V_{x}|$.
We then set $V_{x0}=U_{1}$ and $V_{x1}=U_{2}$, which we see
satisfy (\ref{eq:biclique-skeleton-size-lb}) inductively. We continue
in this way until we define $V_{x}$ for every $x\in C_{2}$. To conclude,
suppose that $y,z\in C_{2}$ are adjacent under $\rho_2$, and let $x$ be their
longest common prefix. Then $V_{y}\subseteq V_{x0}$ and $V_{z}\subseteq V_{x1}$,
and we know that every edge between $V_{x0}$ and $V_{x1}$
is oriented from $V_{x0}$ to $V_{x1}$, which implies that
$(V_{y},V_{z})$ is $1$-dense, as claimed. 
\end{proof}

\subsection{The inner stage}\label{subsec:inner-stage}
In this subsection, we will see how to use the various structures built in the previous subsection to successfully embed a copy of $H$ in any sufficiently large tournament $T$. The basic idea is that the $(\rho,\delta)$-skeletons constructed in Lemmas \ref{lem:build-skeleton} and \ref{lem:biclique-skeleton} are precisely the structures we need in order to apply Lemma \ref{lem:inner-stage} and find a copy of $H$. 

Recall that a prefix coloring is a prefix labeling where the preimage
of every codeword is an independent set. Our first result here shows a general
upper bound on $\ovar{r_{1}}(H)$ in terms of the depth and complexity of
a prefix coloring of $H$.
\begin{thm}\label{thm:general-rho}
Let $H$ be an acyclic digraph on $n$ vertices with maximum degree
$\Delta\geq1$. Then for any prefix coloring $\rho:V(H)\to C$, we have $\ovar{r_1}(H) \leq N$, where
\[
N = \left(2^{5\Delta+4}\Delta^{2\Delta+1}\right)^{\depth(\rho)}\comp(\rho)^{\Delta}n.
\]
\end{thm}

\begin{proof}
Let $T$ be a tournament on $N$
vertices, and suppose for contradiction that $T$ is $H$-free. By
Lemma \ref{lem:build-skeleton} applied with $d=\Delta$ and $c=\frac{1}{8\Delta^2}$, we can find
in $T$ a $(\rho,1-\frac{1}{8\Delta^{2}})$-skeleton $\{V_{x}\}_{x\in C}$
of size at least 
\[
\left(\frac{(8\Delta^{2})^{-\Delta}}{4^{\Delta+1}\Delta}\right)^{\depth(\rho)}\comp(\rho)^{-\Delta}N=\left(2^{5\Delta+2}\Delta^{2\Delta+1}\right)^{-\depth(\rho)}\comp(\rho)^{-\Delta}N\geq4n.
\]
Now, for any vertex $v_{i}\in V(H)$, let $V_{i}=V_{\rho(v_{i})}$.
Then since $\rho^{-1}(x)$ is an independent set for any $x\in C$,
we see that if $v_{i}\to v_{j}$ is an edge of $H$, then $(V_{i},V_{j})$
is a $(1-\frac{1}{8\Delta^{2}})$-dense pair. Therefore, by Lemma
\ref{lem:inner-stage}, we conclude that $T$ contains a copy of $H$. 
\end{proof}
The second result of this subsection uses the $(\rho,1)$-skeletons we constructed in case we are given a forest prefix labeling. Using this skeleton, we are able to prove the following result, which takes as input a forest prefix labeling and any arbitrary prefix labeling (which may be the same as the forest prefix labeling). The output is a bound on the Ramsey number, in terms of the depth and complexity of the first labeling, and in terms of the maximum Ramsey number of the parts in the partition induced by the second prefix labeling.
\begin{thm}
\label{thm:embed-two-labelings}Let $d \geq 2$, and let $H$ be a $d$-degenerate acyclic
digraph on $n$ vertices with maximum degree $\Delta$. Suppose that
$\rho_{1}:V(H)\to C_{1}$ is a forest prefix labeling of $H$ and that $\rho_{2}:V(H)\to C_{2}$
is any prefix labeling. Let $A_1$ and $m_1$ be the maximum degree and maximum component size of $\rho_1$, respectively. Then $\ovar{r_1}(H) \leq N$, where
\[
N=\left(m_1(2^{10}A_{1}\Delta^{2})^{d\depth(\rho_{1})}\comp(\rho_{1})^d\right)^{\depth(\rho_{2})}\max\left(n,\max_{x\in C_{2}}\ovar{r_{1}}(H[\rho_2^{-1}(x)])\right).
\]
 
\end{thm}

\begin{proof}
Let $T$ be a tournament on $N$ vertices, and suppose that $T$ is $H$-free. We apply Lemma \ref{lem:biclique-skeleton}
to $T$, which allows us to find a $(\rho_{2},1)$-skeleton $\{V_{x}\}_{x\in C_{2}}$
where
\[
|V_{x}|\geq\max_{y\in C_{2}}\ovar{r_{1}}(H[\rho_2^{-1}(y)])\geq \ovar{r_{1}}(H[\rho_2^{-1}(x)])
\]
for all $x\in C_{2}$. In other words, this is a collection of disjoint sets $\{V_x\}_{x \in C_2}$ such that if $x$ precedes $y$ in the lexicographic ordering, then every edge is oriented from $V_x$ to $V_y$. By the definition of the Ramsey number, we see that the induced subtournament
$T[V_{x}]$ must contain a copy of $H[\rho_2^{-1}(x)]$ for all $x \in C_2$. We pick such
a copy arbitrarily for each $x\in C_{2}$, and observe that their union
forms a copy of $H$ in $T$. 
\end{proof}
\section{Upper bounds on oriented Ramsey numbers}\label{sec:upper-bounds}

\subsection{Upper bounds in terms of height}

Theorem \ref{thm:general-rho}, which bounds $\ovar{r_{1}}(H)$ in terms of the dyadic complexity and
depth of a prefix coloring, allows us to prove bounds on $\ovar{r_{1}}(H)$ in terms of other,
more natural, parameters. For instance, the next lemma relates the
height of an acyclic digraph to its dyadic complexity and depth.
\begin{lem}\label{lem:height-implies-rho}
If $H$ is an acyclic digraph of height $h$, then there exists a
prefix coloring $\rho:V(H)\to C$ with $\depth(\rho)\leq\lceil\log h\rceil$
and $\comp(\rho)\leq h^{\lceil\log h\rceil}$. 
\end{lem}

\begin{proof}
Recall that if $H$ has height $h$, then $H$ has a directed partition
$P_{0},\ldots,P_{h-1}$ into $h$ independent sets. This partition
naturally yields a prefix coloring using the prefix code $C$ consisting
of all binary strings of length $\lceil\log h\rceil$. Namely, we can label each vertex in $P_{i}$
by the base-$2$ representation of $i$, and this yields a prefix
coloring with depth $\lceil\log h\rceil$. To estimate the dyadic
complexity of this prefix coloring, note that for any binary string
$x$ of length $\ell\leq\lceil\log h\rceil$, we have that $a(x)\leq2^{\lceil\log h\rceil-\ell}$,
since there are at most $2^{\lceil\log h\rceil-\ell}$ codewords in
$C$ prefixed by $x$. Therefore,
\[
\comp(\rho)\leq\prod_{\ell=0}^{\lceil\log h\rceil}2^{\lceil\log h\rceil-\ell}=2^{\lceil\log h\rceil^{2}-\sum_{\ell=0}^{\lceil\log h\rceil}\ell}\leq h^{\lceil\log h\rceil}.\qedhere
\]
\end{proof}
Then Theorem \ref{thm:height-ub} follows as a simple corollary.
\begin{proof}[Proof of Theorem \ref{thm:height-ub}] The result is immediate if $\Delta=1$, so assume $\Delta\geq 2$.
Let $\rho$ be the prefix coloring from Lemma \ref{lem:height-implies-rho}, which has $\depth(\rho)\leq\lceil\log h\rceil$
and $\comp(\rho)\leq h^{\lceil\log h\rceil}$. By Theorem \ref{thm:general-rho},
\[
    \ovar{r_1}(H)\leq \left(2^{5\Delta+4} \Delta^{2\Delta+1}\right)^{\depth(\rho)}\comp(\rho)^\Delta n \leq 2^{7\Delta \log h} \Delta^{3\Delta \log h} h^{2\Delta \log h}n\leq (\Delta h)^{10\Delta \log h}n.\qedhere 
\]
\end{proof}

Similarly, given a graded digraph, one can find a prefix coloring with small depth and dyadic complexity.
\begin{lem}\label{lem:graded-implies-rho}
If $H$ is a graded acyclic digraph of height $h$, then there exists
a prefix coloring $\rho:V(H)\to C$ with $\depth(\rho)\leq\lceil\log h\rceil$
and $\comp(\rho)=1$.
\end{lem}

\begin{proof}
The proof is identical to that of Lemma \ref{lem:height-implies-rho}, except that
$a(x)\leq1$ for every binary string $x$, since the only codeword
prefixed by $x0$ that can have edges to a codeword prefixed by
$x1$ is the codeword $(x,0,1,1,\dotsc,1)$. 
\end{proof}
As before, Theorem \ref{thm:graded} follows as a corollary.
\begin{proof}[Proof of Theorem \ref{thm:graded}] We may again assume that $\Delta\geq 2$.
Let $\rho$ be the prefix coloring from Lemma \ref{lem:graded-implies-rho}, with $\depth(\rho)\leq\lceil\log h\rceil$
and $\comp(\rho)=1$. Then from Theorem \ref{thm:general-rho},
\[
    \ovar{r_1}(H)\leq \left(2^{5\Delta+4} \Delta^{2\Delta+1}\right)^{\depth(\rho)}\comp(\rho)^\Delta n \leq 2^{7\Delta \log h} \Delta^{3\Delta \log h}  n\leq h^{10\Delta \log \Delta} n.\qedhere 
\]
\end{proof}
Recall from the introduction that the bandwidth of an $n$-vertex acyclic digraph $H$ is the least $\ell$ so that $P_n^\ell$ contains $H$. Using the same argument, one can obtain a bound of $\ovar{r_1}(H)\leq n^{O_\ell(1)}$ for any $n$-vertex acyclic digraph of bandwidth at most $\ell$, using the fact that the same binary-representation prefix coloring have dyadic complexity at most $n^{O(\log \ell)}$. However, since the Ramsey number of bounded-bandwidth acyclic digraphs is known to be linear \cite{Draganicetal}, we omit the proof of this weaker result.

\subsection{Upper bounds for random digraphs}

In this section, we prove Theorem \ref{thm:G(n,d)-almost-linear},
showing that if $d\ge1$ is bounded, then w.h.p.\ 
$\ovar{r_{1}}(H)$ is nearly linear when $H=\overrightarrow{G}(n,d)$. Additionally, we prove a nearly-linear upper bound when $p=d/n$ and $H=\overrightarrow{G}(n,p)$. 

We will need the following result of Dross and Havet \cite{DrHa} mentioned in the introduction.
They only state this result for orientations of trees, but one
can easily extend it to orientations of forests by adding edges to
join distinct connected components.
\begin{thm}
\label{thm:dross-havet}Let $H$ be a $1$-degenerate digraph (i.e.\
an orientation of an undirected forest) on $n\ge 2$ vertices. If $T$
is any tournament on at least $\frac{21}{8}n-\frac{47}{16}$ vertices, then $T$ contains a copy of
$H$.
\end{thm}
We recall that the \emph{vertex
arboricity} of an undirected graph is the minimum number of subsets
needed to cover the vertices of the graph, such that each subset induces
a forest (see e.g.\ \cite{Burr} for more on this undirected graph parameter). We first define a natural directed analogue of this quantity, though for technical reasons we also keep track of the maximum component size in such a partition.
\begin{defn}
The \emph{directed vertex arboricity} $\va(H)$ of an acyclic digraph
$H$ is the minimum size $r$ of a directed partition of $H$ into
sets $\{P_{i}\}_{i=1}^{r}$ where $H[P_{i}]$ is $1$-degenerate for
all $i$. 

Additionally, if $H$ has a directed partition into sets $\{P_i\}_{i=1}^r$ such that $H[P_i]$ is $1$-degenerate for all $i$, and every weakly connected component in $H[P_i]$ has at most $m$ vertices, then we say that $H$ has an \emph{$(r,m)$-forest partition}.
\end{defn}

Both of our upper bound results for random digraphs follow from the following
theorem, which yields an upper bound on $\ovar{r_{1}}(H)$ in terms of the
degeneracy, maximum degree, and forest partition size of $H$. 
\begin{thm}
\label{thm:general-arboricity}Let $H$ be an acyclic digraph on $n$
vertices with maximum degree $\Delta$, degeneracy $d\geq 2$, and with an $(r,m)$-forest partition. Then $\ovar{r_{1}}(H)\leq(r\Delta)^{6d(\log r)^{2}} m^{2 \log r} n$.
\end{thm}

\begin{proof}
Let $s=\lceil\log r\rceil$, and let $C$ denote the prefix code consisting
of all strings of length $s$. Fix a partition $P_{0}\sqcup\dotsb\sqcup P_{r-1}$
of $V(H)$ into directed forests such that every edge between $P_{i}$
and $P_{j}$ is oriented from $P_{i}$ to $P_{j}$ for all $i<j$, and such that $H[P_i]$ has weakly connected components of size at most $m$.
Let $\rho:V(H)\to C$ be the forest prefix labeling mapping $P_{i}$
to the binary representation of $i$. Then the maximum degree
of $\rho$ is at most $|C|=2^s \leq 2r$, the maximum component size of $\rho$ is at most $m$, and $\depth(\rho)=s$. Moreover, we
can bound the dyadic complexity of $\rho$ as $\comp(\rho)\leq r^{s}$,
since every binary string is the prefix of at most $r$ codewords
in $C$. We now now apply Theorem \ref{thm:embed-two-labelings} with
$\rho_{1}=\rho_{2}=\rho$. We recall that since $\rho^{-1}(x)$ is
a forest for every $x$, we have that $\ovar{r_{1}}(H[\rho^{-1}(x)])\leq3n$
by Theorem \ref{thm:dross-havet}. Therefore, 
\begin{align*}
\ovar{r_{1}}(H)&\leq\left(m_1(2^{10}A_1\Delta^{2})^{d\depth(\rho_{1})}\comp(\rho_{1})^d\right)^{\depth(\rho_{2})}\max\left(n,\max_{x\in C_{2}}\ovar{r_{1}}(H[\rho_2^{-1}(x)])\right)\\
&\leq\left(m(2^{11}r\Delta^{2})^{ds}r^{ds}\right)^{s}\cdot(3n)\\
&\leq (r\Delta)^{6d(\log r)^{2}} m^{2 \log r} n. \qedhere
\end{align*}
\end{proof}
It remains to check that both $\overrightarrow{G}(n,d)$ and $\overrightarrow{G}(n,p)$
satisfy the conditions of Theorem \ref{thm:general-arboricity} for
appropriate $\Delta,d,m,$ and $r$. Maximum degree and degeneracy are both easy to control for these graphs, so the nontrivial
part is bounding their directed vertex arboricity, or more precisely the parameters of a forest partition.

Consider the random digraph $\overrightarrow{G}(n,p)$, where $np=d\ge1$. The idea for
bounding $\va(H)$ is to equitably divide $[n]$ into $5d$ intervals
$I_{i}$, and note that $H[I_{i}]\sim\overrightarrow{G}(n/5d,p)$
is in the subcritical regime of the Erd\H os--R\' enyi random graph process,
so w.h.p.\ is a union of trees and unicyclic components, each comprising $O(\log n)$ vertices. Each interval
$I_{i}$ can be further divided in two to break the unicyclic components,
which shows that w.h.p.\  we have a $(10d, O(\log n))$-forest partition. The analysis is broadly
similar for $H=\overrightarrow{G}(n,d)$ but somewhat more technical.

Recall that if $d\ge1$ and $nd$ is even, a uniformly random undirected
$d$-regular graph $G(n,d)$ can be generated using the ``pairings
model'' (also known as the ``configurations model'') of Bollob\' as
\cite{Bollobas}, see the survey of Wormald \cite{Wormald}. The pairings
model of $G(n,d)$ generates a random $d$-regular multi-graph (with
self-loops allowed) by taking a uniformly random perfect matching
(a ``pairing'') on $nd$ points divided into $n$ $d$-sets, and
then contracting each $d$-set into a single vertex. There are a total
of
\[
P(nd)=\frac{(nd)!}{(nd/2)!\cdot2^{nd/2}}
\]
such pairings, and any simple $d$-regular graph is equally likely
to be generated. If $d\ge1$ is fixed and $n\rightarrow\infty$, it
is known that the probability a pairing generates a simple graph is
asymptotic to $e^{(1-d^{2})/4}$. To generate an honest $G(n,d)$,
repeatedly sample from the above model (expected constant number of
samples) until a simple graph is found.
\begin{lem}
\label{lem:bounded-arboricity-G(n,d)}If $d\ge2$ is fixed, $nd$
is even, and $G$ is the induced subgraph of $G(n,d)$ on a fixed
set of $\frac{n}{5d}$ vertices, then w.h.p.\ every connected component
of $G$ has order at most $2\log n$ and contains at most one cycle.
\end{lem}

\begin{proof}
We first show that w.h.p.\ each component contains at most one cycle.

It is not difficult to show that any \emph{minimal} graph $H$
on $k$ vertices with at least two cycles is formed from a path of
length $k$ by adding edges from each of its ends, see e.g.\ \cite[Theorem 5.5]{JaLuRu}. Thus, there are at most $k^{2}\cdot k!$
labelled graphs on $k$ vertices with this property. For a fixed such
$H$, we bound the probability it appears among $k$ fixed vertices
$v_{1},\ldots,v_{k}$ in the pairings model for $G(n,d)$. The total
number of pairings giving such an $H$ (without giving rise to multi-edges)
can be bounded above by
\[
d^{2(k+1)}\cdot P(nd-2(k+1)),
\]
since there are at most $(d^{2})^{k+1}$ ways to choose the edges
between $d$-sets that correspond to the $k+1$ edges of $H$, and
then at most $P(nd-2(k+1))$ ways to pair the remaining points. Let $\mathcal E$ be the event that the pairings model generating a simple graph, which occurs with probability $(1+o(1))e^{(1-d^2)/4}$. We get that
\begin{align*}
\Pr[\{v_{1},\ldots,v_{k}\}\text{ is a copy of }H \mid \mathcal E]&\le(1+o(1))e^{-(1-d^{2})/4}\cdot\frac{d^{2(k+1)}P(nd-2(k+1))}{P(nd)}\\
&\le(1+o(1))e^{-(1-d^{2})/4}\cdot\frac{d^{k+1}}{n^{k+1}}.
\end{align*}
Thus,
\[
\Pr[\text{some such }H\text{ appears in }G]\le(1+o(1))e^{-(1-d^{2})/4}\sum_{k\ge4}\binom{n/(5d)}{k}k^{2}\cdot k!\cdot\frac{d^{k+1}}{n^{k+1}}=o(1),
\]
which completes the proof that every component of $G$ contains at most one cycle w.h.p.

Next, we show that w.h.p.\ every connected component has order at most $2\log n$ with a similar computation. For a given set size $k$, the number of labelled trees $H$ on $k$ vertices is $k^{k-2}$ by Cayley's theorem. We bound the probability that a fixed such $H$ appears among $k$ fixed vertices $v_1,\ldots, v_k$. Similarly to before, we obtain that the total number of pairings giving such an $H$ without multi-edges is bounded above by
\[
d^{2(k-1)} \cdot P(nd-2(k-1)),
\]
since there are at most $(d^2)^{k-1}$ ways to choose the edges between $d$-sets that correspond to the $k-1$ edges of the tree $H$, and at most $P(nd-2(k-1))$ ways to pair the remaining points. Conditioning on the pairings model generating a simple graph, which occurs with probability $(1+o(1))e^{(1-d^2)/4}$, we get
\begin{align*}
\Pr[\{v_{1},\ldots,v_{k}\}\text{ is a copy of }H \mid \mathcal E]&\le(1+o(1))e^{-(1-d^{2})/4}\cdot\frac{d^{2(k-1)}P(nd-2(k-1))}{P(nd)}\\
&\le(1+o(1))e^{-(1-d^{2})/4}\cdot\frac{d^{k-1}}{n^{k-1}}.
\end{align*}
Taking a union bound over all choices of $v_1,\ldots, v_k$,
\begin{align*}
\Pr[\text{a tree on } k \text{ vertices appears in }G]&\le(1+o(1))e^{-(1-d^{2})/4}\binom{n/(5d)}{k} k^{k-2}\cdot \frac{d^{k-1}}{n^{k-1}} \\
& \le (1+o(1))e^{-(1-d^{2})/4}\frac{n^k\cdot k^{k-2} \cdot d^{k-1}}{(5d)^k \cdot (k/e)^k \cdot n^{k-1}} \\
& \le n \cdot (e/5+o(1))^k,
\end{align*}
which is $o(1)$ for $k =2\log n$. This implies that w.h.p.\ every component of $G$ has order at most $2\log n$, and completes the proof.
\end{proof}
We can now prove Theorem \ref{thm:G(n,d)-almost-linear}.
\begin{cor}
For any $d\geq2$, we have that w.h.p., $\ovar{r_{1}}(\overrightarrow{G}(n,d))\leq n(\log n)^{4 \log d}$ as $n \to \infty$.
\end{cor}

\begin{proof}
By Lemma \ref{lem:bounded-arboricity-G(n,d)}, we know that w.h.p.,
any fixed $n/(5d)$ vertices of $\overrightarrow{G}(n,d)$ span a
disjoint union of trees and unicyclic components, each of which has order at most $2 \log n$. Therefore, by applying
the union bound to $5d$ events, we conclude that w.h.p., $\overrightarrow{G}(n,d)$
has a directed partition into $5d$ parts all with this property.
Dividing each part in two to split every cycle, we conclude that w.h.p.\ $\ovar G(n,d)$ has a $(10d, 2\log n)$-forest partition. Additionally, since the underlying
undirected graph of $\overrightarrow{G}(n,d)$ is $d$-regular, we
see that $\overrightarrow{G}(n,d)$ is $d$-degenerate and has maximum
degree $d$. So by Theorem \ref{thm:general-arboricity}, we conclude
that w.h.p.\ as $n\to \infty$,
\[
\ovar{r_{1}}(\overrightarrow{G}(n,d))\leq (10d^2)^{6d(\log(10 d))^{2}}(2\log n)^{2\log(10d)}n,
\]
which is upper-bounded by $n (\log n)^{4\log d}$ for any fixed $d$ and sufficiently large $n$.
\end{proof}

Let $\ovar G(n,p)$ denote the orientation of the Erd\H os--R\'enyi random graph $G(n,p)$ on vertex set $[n]$ where all edges are oriented to the right. Similarly to the above, we can prove a nearly-linear upper bound on $\ovar{r_1}(\ovar G(n,p))$.
\begin{cor}\label{cor:G(n,p)-almost-linear}
For any $d\geq2$, we have that w.h.p., $\ovar{r_{1}}(\overrightarrow{G}(n,p))\leq n\cdot(\log n)^{O(d(\log d)^{2})}$,
where $p=d/n$.
\end{cor}

\begin{proof}
It is easy to show that if $p=d/n$, then $G(n,p)$ is $O(d)$-degenerate (see \cite[Theorem 4.8]{FoSu} for a proof of a stronger result).
Additionally, it is well-known that the maximum degree of $G(n,p)$
is $O(\log n/\log \log n)$ for any fixed $d\geq2$. Finally, by the easier version
of Lemma \ref{lem:bounded-arboricity-G(n,d)} (e.g.\ \cite[Theorem 5.5]{JaLuRu}), we see that $\overrightarrow{G}(n,p)$ has a $(10d, 2\log n)$-forest partition.
Hence, Theorem \ref{thm:general-arboricity} implies that
$\ovar{r_{1}}(\overrightarrow{G}(n,p))\leq n\cdot(\log n)^{O(d(\log d)^{2})}$.
\end{proof}

\section{Multiple colors and ordered Ramsey numbers}\label{sec:multicolor&ordered}

In this section we study oriented Ramsey numbers of more than one
color, proving Theorems \ref{thm:multicolor-lower-bound} and \ref{thm:multicolor-upper-bound}. 

We define an \emph{ordered graph} $G$ to be an undirected graph whose
vertex set comes with a total order. If the vertex set is a subset
of $\mathbb{N}$, then the total order is assumed to be that of $\mathbb{N}$.
If $G_1,\ldots, 
G_k$ are ordered graphs, then the \emph{ordered Ramsey
number} $r_{<}(G_1,\ldots, G_k)$ is the minimum $N$ such that any edge-coloring
of the complete graph on $[N]$ in colors $1,\ldots, k$ contains a monochromatic copy of $G_i$ in color $i$ for some $i$. Here
an \emph{ordered copy} of $G_{i}$ is a subgraph isomorphic to $G_{i}$
with vertices appearing in the same order. We write $r_{<}(G;k) \coloneqq r_{<}(G,\ldots, G)$ when all the graphs $G_i$ are the same.

\subsection{The lower bound}

To prove Theorem \ref{thm:multicolor-lower-bound}, we need the following
theorem from Conlon,
Fox, Lee, and Sudakov \cite[Theorem 2.3]{CoFoLeSu}.
\begin{thm}
\label{thm:scrambled-matchings}If $M$ is a random matching on vertex set $[n]$, then w.h.p.,
\[
r_{<}(M;2)>n^{\log n/20\log\log n}.
\]
\end{thm}

We remark that the existence of ordered matchings $M$ with $r_<(M;2) > n^{\log n/20\log \log n}$ was proven independently by Balko, Cibulka, Kr\'al, and Kyn\v{c}l~\cite{BaCiKrKy}. We will only need this weaker result.

If $H$ is an acyclic digraph with a Hamiltonian (directed) path,
we say $H$ is \emph{Hamiltonian}. It has a unique vertex ordering
$v_{1},\ldots,v_{n}$ where consecutive vertices are adjacent and all edges point forwards. We assign to
$H$ a natural ordered graph $H^{+}$ on $[n]$ where $i\sim j$ if and only if $v_{i}\rightarrow v_{j}$
in $H$.
\begin{lem}
\label{lem:hamiltonian-to-ordered}If $k\ge 1$ and $H$ is an acyclic Hamiltonian
digraph, then
\[
\ovar{r_{k}}(H)\ge r_{<}(H^{+}; k).
\]
\end{lem}

\begin{proof}
If $N=r_{<}(H^+;k)-1$, there exists a $k$-edge-coloring $\chi$ of the complete graph on $[N]$
in which there is no monochromatic copy of $H^+$. Let $T$ be the transitive tournament on $[N]$ with all edges oriented forwards (i.e.\ $i\to j$ if and only if $i<j$), and the edge
between $i \to j$ also colored $\chi(i,j)$. Since $H$ has a Hamiltonian
path and all edges in $T$ point forwards in $[N]$, any copy of $H$ in $T$ corresponds to an ordered copy of $H^{+}$ in $\chi$. By construction, there is no monochromatic copy of $H$ in $T$, as desired.
\end{proof}
The two results above together can be used to prove Theorem \ref{thm:multicolor-lower-bound}.
\begin{proof}[Proof of Theorem \ref{thm:multicolor-lower-bound}]
Since $\ovar{r_k}(H)$ is nondecreasing in $k$, it suffices to prove the result for $k=2$. By Theorem \ref{thm:scrambled-matchings}, there exists a matching $M$ on $[n]$ which satisfies
$r_{<}(M;2)>n^{\log n/20\log\log n}$. Define $H$ to be the acyclic digraph on $[n]$ where $i\rightarrow j$
if $i<j$ and either $j=i+1$ or $(i,j)$ is an edge of $M$. By Lemma
\ref{lem:hamiltonian-to-ordered}, $\ovar{r_{2}}(H)\ge r_{<}(H^{+};2)$.
On the other hand, $M$ is a ordered subgraph of $H^{+}$, so $r_{<}(H^{+}; 2)\ge r_{<}(M;2)$.
It follows that
\[
\ovar{r_{2}}(H)\ge r_{<}(M;2)>n^{\log n/20\log\log n},
\]
as desired. Note that $H$ is the edge union of a path
and a matching, so it has maximum degree $3$.
\end{proof}

\subsection{The upper bound}

To prove Theorem \ref{thm:multicolor-upper-bound}, we upper bound
$k$-color oriented Ramsey numbers by $2k$-color ordered Ramsey numbers. Let $H^-$ be the ordered graph obtained from $H^+$ by reversing the vertex order.
\begin{lem}
\label{lem:ordered-to-oriented}If $k\ge1$ and $H$ is an acyclic
digraph, then
\[
\ovar{r_{k}}(H)\le r_{<}(\underbrace{H^{+},\ldots,H^{+}}_{k},\underbrace{H^{-},\ldots,H^{-}}_{k}).
\]
\end{lem}

\begin{proof}
Let $T$ be a tournament on 
\[
N=r_{<}(\underbrace{H^{+},\ldots,H^{+}}_{k},\underbrace{H^{-},\ldots,H^{-}}_{k})
\]
vertices, with an edge-coloring $\chi$ using colors $1,\dotsc,k$. Arbitrarily identify $V(T)$
with $[N]$ and define a $(2k)$-edge-coloring $\chi'$ of $K_{N}$
by 
\[
\chi'(i,j)=\begin{cases}
\chi(i,j) & \text{if }i<j\text{ and }i\rightarrow j\\
\chi(i,j)+k & \text{else}.
\end{cases}
\]

By the definition of $N$, there is either some color $c\le k$ where
$\chi'$ has an ordered copy of $H^{+}$ in color $c$, or some color
$c>k$ where $\chi'$ has an ordered copy of $H^{-}$ in $c$. In
the former case, $T$ contains a monochromatic copy of $H$ in color
$c$ with all vertices pointed forwards in the arbitrary ordering,
and in the latter case $T$ contains a monochromatic copy of $H$
in color $c-k$ with all edges pointed backwards. In either case,
$T$ contains a monochromatic copy of $H$, as desired.
\end{proof}
To prove an upper bound on $\ovar{r_{k}}(H)$, it remains to generalize the
following upper bound of \cite{CoFoLeSu} on $2$-color ordered Ramsey numbers
to more colors.

\begin{thm}[{\cite[Theorem 3.6]{CoFoLeSu}}]
\label{thm:2-color-ordered}  If $H$ is an ordered graph on at most $n$ vertices with degeneracy $d\ge 2$, then
\[
r_{<}(H,K_n) \le 2^{O(d\log ^2(2n/d))}.
\]
\end{thm}

The multicolor bound can now be obtained by iterating the above theorem.

\begin{thm}
\label{thm:multi-color-ordered}If $k,d\ge2$ and $H_{1},\ldots,H_{k-1}$
are $d$-degenerate ordered graphs on at most $n$ vertices, then
\[
r_{<}(H_{1},\ldots,H_{k-1},K_{n})\le2^{O_{k,d}(\log^{2^{k-1}}n)}.
\]
\end{thm}

\begin{proof}
We prove the theorem by induction on $k$. The base case $k=2$ is just Theorem~\ref{thm:2-color-ordered}. For the inductive step, note that if $M= r_<(H_2,\ldots, H_{k-1}, K_n) \le 2^{O(\log^{2^{k-2}} n)}$ then one obtains
\[
r_<(H_1, H_2, \ldots, H_{k-1}, K_n) \le r_<(H_1, K_M),
\]
by combining the last $k-1$ colors into one ``super-color.'' It follows by applying the base case that
\[
r_<(H_1, H_2, \ldots, H_{k-1}, K_n) \le r_<(H_1, K_M) \le 2^{O(\log ^2 M)} \le 2^{O(\log^{2^{k-1}}n)},
\]
as desired.
\end{proof}
Theorem \ref{thm:multicolor-upper-bound} follows by combining Lemma
\ref{lem:ordered-to-oriented} with Theorem \ref{thm:multi-color-ordered}.

\section{Concluding Remarks}\label{sec:concluding}

In this section we collect a few appealing open problems on the Ramsey numbers of digraphs. For $k,\Delta \ge 1$ and $n > \Delta$, let $H_{k,n,\Delta}$ be an acyclic digraph $H$ with $n$ vertices and maximum degree $\Delta$ maximizing the value of $\ovar{r_k}(H)$. Much of this paper was devoted to understanding the growth rate of $\ovar{r_k}(H_{k,n,\Delta})$ for fixed $k$ and $\Delta$.

We first consider the one-color case. Theorem~\ref{thm:general-lower-bound}, Lemma~\ref{lem:ordered-to-oriented}, and Theorem~\ref{thm:2-color-ordered} together show
\begin{equation}\label{eq:one-color-state}
n^{\Omega(\Delta^{2/3}/\log^{5/3}\Delta)}\le \ovar{r_1}(H_{1,n,\Delta}) \le 2^{O(\Delta\log ^2 (2n/\Delta))}
\end{equation}
While we do not know whether the above Ramsey number grows polynomially or super-polynomially in $n$ for $k=1$, we also showed that for $k\ge 2$ and $\Delta \ge 3$, $\ovar{r_k}(H_{k,n,\Delta})$ is at least $n^{\Omega(\log n/\log \log n)}$. We conjecture that a super-polynomial growth rate is also possible for $k=1$.
\begin{conj}\label{conj:super-poly}
There exists an absolute constant $\Delta$ such that
\[
\ovar{r_{1}}(H_{1,n,\Delta})\ge n^{\omega(1)}.
\]
\end{conj}

In the case of $k\ge 2$ colors, Theorems~\ref{thm:multicolor-lower-bound} and~\ref{thm:multicolor-upper-bound} together imply
\begin{equation}\label{eq:multi-color-state}
n^{\Omega(\log n/\log\log n)} \le \ovar{r_{k}}(H_{k,n,\Delta}) \le 2^{O(\log^{2^{2k-1}} n)}.
\end{equation}
It would be interesting to determine even the logarithmic order of $\ovar{r_{k}}(H_{k,n,\Delta})$.

\begin{problem} \label{prob:log-order}
For any fixed $k \ge 1$ and $\Delta\ge 2$, determine the order of growth of $\log \ovar{r_{k}}(H_{k,n,\Delta})$.
\end{problem}

To our knowledge, the only solved case of Problem~\ref{prob:log-order} is $k=1$, $\Delta = 2$. Here, $H_{1,n,2}$ must be a disjoint union of arbitrarily oriented paths and non-strongly oriented cycles (an acyclic $H$ cannot contain a strongly oriented cycle). Thomason~\cite{Thomason} proved that if $n$ is large enough, $\ovar{r_1}(C_n) = n$ for any non-strongly oriented cycle $C_n$ on $n$ vertices, which can be used to show $\ovar{r_{1}}(H_{1,n,2}) = n+O(1)$.

It would be interesting to improve either side of (\ref{eq:multi-color-state}) substantially. We tentatively conjecture that neither side is close to the truth, in the following quantitative form.

\begin{conj}
There exist $\Delta \ge 3$ and constants $c_k = \omega_k(1)$ and $C_k = 2^{o(k)}$ such that
\[
\Omega(\log ^{c_k} n) \le \log \ovar{r_{k}}(H_{k,n,\Delta}) \le O(\log ^{C_k} n).
\]
\end{conj}

In the one-color case, another open problem we find interesting is to determine more precisely the Ramsey number of a sparse random digraph. Theorem \ref{thm:G(n,d)-almost-linear} shows that for fixed $d$, the Ramsey number of $\ovar G(n,d)$ is w.h.p.\ bounded above by $n(\log n)^{O_d(1)}$. We expect that the answer is in fact linear. 

\begin{conj}
\label{conj:random}
If $d\ge 2$ is fixed and $H = \ovar{G}(n,d)$, then w.h.p. $\ovar{r_1}(H) = O_d (n).$
\end{conj}

This would follow from our techniques if one could prove the following strengthening of Lemma \ref{lem:one-biclique-basic}, in which the size of the $1$-dense pair depends only on the maximum degree $\Delta$ (and not on the number $m$ of vertices).
\begin{conj}
\label{conj:better-biclique}
For every $\Delta \geq 1$, there exists $C_\Delta>0$ such that the following holds. Let $H$ be a $1$-degenerate digraph with
maximum degree $\Delta$ on $m$ vertices $v_{1},\ldots,v_{m}$, and
let $T$ be an arbitrary tournament. If there exist sets $U_{1},\ldots,U_{m}\subseteq V(T)$, each
of size $M\geq C_\Delta m$, such that there is no embedding $\phi:H\hookrightarrow T$
satisfying $\phi(v_{i})\in U_{i}$, then $T$ contains a $1$-dense
pair with size at least $M/C_\Delta$.
\end{conj}
Indeed, if one could prove Conjecture \ref{conj:better-biclique}, then it would imply that Lemma \ref{lem:one-biclique-again}, Lemma \ref{lem:biclique-skeleton}, Theorem \ref{thm:embed-two-labelings}, and Theorem \ref{thm:general-arboricity} would all no longer depend on $m$, the maximum size of a tree component in a directed partition of $H$ into oriented forests. In particular, this would imply Conjecture~\ref{conj:random}. Moreover, it is entirely possible that Conjecture~\ref{conj:better-biclique} could be proven using similar greedy embedding arguments, since the conjecture is not hard to prove if the injectivity constraint on $\phi$ is removed. 

We would also like to highlight one powerful digraph embedding technique that we have not used in this paper, but are hopeful can be incorporated into our arguments to prove more general results. The {\it median ordering} of a tournament $T$ is the vertex ordering $v_1,\ldots,v_n$ maximizing the number of forward edges. To see the power of the median ordering, note that $v_i\rightarrow v_{i+1}$ for every $1\le i\le n-1$ in this ordering, so this immediately exhibits a Hamiltonian path in $T$. Previous work showing linear upper bounds on $\ovar{r_1}(H)$ when $H$ is an oriented tree (e.g.\ \cite{ElSahili}) or an acyclic digraph of bounded bandwidth~\cite{Draganicetal} all depend on embedding $H$ into a tournament $T$ in some iterative manner according to its median ordering. We were not able to reproduce these upper bounds using greedy embedding arguments, which seem primarily suited for embedding digraphs $H$ without long paths.

Finally, we recall from the introduction the \emph{directed Ramsey number} $\overleftrightarrow{r_k}(H)$, which is defined as the least $N$ such that every $k$-coloring of the edges of $\overleftrightarrow{K_N}$ contains a monochromatic copy of $H$. It is easy to see that $\ovar{r_k}(H) \geq \overleftrightarrow{r_k}(H)$ for any $k$ and any acyclic $H$. Indeed, if $N=\ovar{r_k}(H)$, then given a $k$-edge-coloring of $\overleftrightarrow{K_N}$, we may ignore one edge from each anti-parallel pair to obtain a $k$-edge-colored $N$-vertex tournament, which contains a monochromatic $H$. There is also an inequality in the other direction, whose proof is identical to that given in Lemma \ref{lem:ordered-to-oriented}, and which states that $\ovar{r_k}(H) \leq \overleftrightarrow{r_{2k}}(H,\dots,H, H',\dots,H')$, where there are $k$ copies of $H$ and $k$ of $H'$, and $H'$ is obtained from $H$ by reversing all the edges. Thanks to these connections, one can convert many of our results on oriented Ramsey numbers to results on directed Ramsey numbers. For example, Theorem \ref{thm:multicolor-upper-bound} immediately implies a quasi-polynomial upper bound on $\overleftrightarrow{r_k}(H)$ for any bounded-degree acyclic digraph $H$. In the other direction, since reversing the edges of any interval mesh yields another interval mesh, Theorem \ref{thm:general-lower-bound} shows the existence of a bounded-degree acyclic digraph $H$ with $\overleftrightarrow{r_2}(H)$ which grows faster than any fixed polynomial in the number of vertices of $H$. For $k \geq 4$ colors, we can similarly use Theorem \ref{thm:multicolor-lower-bound} to produce a bounded-degree acyclic digraph $H$ such that $\overleftrightarrow{r_k}(H)$ grows super-polynomially. However, there is an interesting intermediate case at $k=3$ colors, and we end with the following conjecture, which may be easier than Conjecture \ref{conj:super-poly}.
\begin{conj}
	There is an absolute constant $\Delta$ and an infinite sequence $\{H_n\}$ of $n$-vertex acyclic digraphs with maximum degree at most $\Delta$ and $\overleftrightarrow{r_3}(H_n) \geq n^{\omega(1)}$.
\end{conj}

\vspace{3mm}

\noindent {\bf Acknowledgments.} We would like to thank Jasmine Yan for producing Figure~\ref{fig:greedy-embedding}, and the anonymous referee for many helpful comments.

\end{document}